\documentclass{imanum}

\jno{drnxxx}
\received{\today}
\revised{}

\usepackage{mathtools}
\usepackage{amsmath,amssymb,amsthm}
\usepackage{pifont}
\newcommand{\xmark}{\ding{55}}%
\newcommand{\cmark}{\ding{51}}%

\usepackage{suffix} 
\usepackage{xpunctuate}
\usepackage{graphicx,graphics,color}
\usepackage[dvipsnames]{xcolor}

\usepackage{booktabs} 
\usepackage{array} 
\newcolumntype{L}[1]{>{\raggedright\let\newline\\\arraybackslash\hspace{0pt}}m{#1}}
\newcolumntype{C}[1]{>{\centering\let\newline\\\arraybackslash\hspace{0pt}}m{#1}}
\newcolumntype{R}[1]{>{\raggedleft\let\newline\\\arraybackslash\hspace{0pt}}m{#1}}

\usepackage{tikz}
\tikzset{
	block/.style={align=center, minimum height=3em},
	arrow/.style={-{Stealth[]}}
}
\usetikzlibrary{positioning,arrows.meta}

\DeclarePairedDelimiter{\abs}{\lvert}{\rvert}
\DeclarePairedDelimiter{\norm}{\lVert}{\rVert}

\newcommand{\ipHsymb}{m}
\DeclarePairedDelimiterX{\ipH}[2]{\ipHsymb\big\lparen}{\big\rparen}{#1, #2} 
\DeclarePairedDelimiterXPP{\normH}[1]{}{\lvert}{\rvert}{}{#1}

\newcommand{\ipHhsymb}{m_{h}}
\DeclarePairedDelimiterX{\ipHh}[2]{\ipHhsymb\big\lparen}{\big\rparen}{#1, #2} 
\DeclarePairedDelimiterXPP{\normHh}[1]{}{\lvert}{\rvert}{_{h}}{#1}
\DeclarePairedDelimiterXPP{\normHhdual}[1]{}{\lVert}{\rVert}{_{\Hh^*}}{#1}

\newcommand{\ipVsymb}{a}
\newcommand{\Vdual}{V^*}
\DeclarePairedDelimiterX{\ipV}[2]{\ipVsymb\big\lparen}{\big\rparen}{#1, #2} 
\DeclarePairedDelimiterXPP{\normV}[1]{}{\lVert}{\rVert}{_V}{#1}
\DeclarePairedDelimiterXPP{\normVa}[1]{}{\lVert}{\rVert}{_{}}{#1}
\DeclarePairedDelimiterXPP{\normVdual}[1]{}{\lVert}{\rVert}{_{ALT,\star}}{#1}
\DeclarePairedDelimiterXPP{\normVadual}[1]{}{\lVert}{\rVert}{_{\star}}{#1}

\DeclarePairedDelimiterX{\ipVh}[2]{\ipVsymb_h\big\lparen}{\big\rparen}{#1, #2} 
\DeclarePairedDelimiterXPP{\seminormVh}[1]{}{\lVert}{\rVert}{_{\sobfhsymb}}{#1}
\DeclarePairedDelimiterXPP{\normVh}[1]{}{\lVert}{\rVert}{_{h}}{#1}
\DeclarePairedDelimiterXPP{\normVhdual}[1]{}{\lVert}{\rVert}{_{\star,h}}{#1}

\newcommand{\sobfsymb}{a} 
\DeclarePairedDelimiterX{\sobf}[2]{\sobfsymb\big\lparen}{\big\rparen}{#1, #2} 
\DeclarePairedDelimiterXPP{\normVsobf}[1]{}{\lVert}{\rVert}{_{\sobfsymb}}{#1}

\newcommand{\sobfhsymb}{\sobfsymb_h} 
\DeclarePairedDelimiterX{\sobfh}[2]{\sobfhsymb\big\lparen}{\big\rparen}{#1, #2} 
\newcommand{\sobfBsymb}{b} 
\DeclarePairedDelimiterX{\sobfB}[2]{\sobfBsymb\big\lparen}{\big\rparen}{#1, #2} 
\newcommand{\sobfBhsymb}{\sobfBsymb_h} 
\DeclarePairedDelimiterX{\sobfBh}[2]{\sobfBhsymb\big\lparen}{\big\rparen}{#1, #2} 
\newcommand{\DeltaipHsymb}{\Delta\ipHsymb} 
\DeclarePairedDelimiterX{\DeltaipH}[2]{\DeltaipHsymb\big\lparen}{\big\rparen}{#1, #2} 
\newcommand{\DeltaipVsymb}{\Delta\ipVsymb} 
\DeclarePairedDelimiterX{\DeltaipV}[2]{\DeltaipVsymb\big\lparen}{\big\rparen}{#1, #2} 
\newcommand{\Deltasobfsymb}{\Delta\sobfsymb} 
\DeclarePairedDelimiterX{\Deltasobf}[2]{\Deltasobfsymb\big\lparen}{\big\rparen}{#1, #2}
\newcommand{\DeltasobfBsymb}{\Delta\sobfBsymb} 
\DeclarePairedDelimiterX{\DeltasobfB}[2]{\DeltasobfBsymb\big\lparen}{\big\rparen}{#1,
	#2}  

\DeclarePairedDelimiterX{\dpV}[2]{\langle}{\rangle_V}{#1,#2}

\DeclarePairedDelimiterXPP{\norminf}[2]{}{\lVert}{\rVert}{_{L^\infty\left( #2 \right)}}{#1}
\DeclarePairedDelimiterXPP{\normsup}[2]{}{\lVert}{\rVert}{_{\infty,#2}}{#1}
\DeclarePairedDelimiterXPP{\col}[1]{}{\lparen}{\rparen}{}{#1}

\newcommand{\lift}[1]{#1^\ell}
\newcommand{\intpo}{\widetilde I_h}
\newcommand{\intpol}{I_h}
\newcommand{\Ritz}{\widetilde R_h}
\newcommand{\Ritzl}{R_h}

\newcommand{\dotr}{\mbox{$\boldsymbol{\cdot}$}}
\newcommand{\defeq}{=}

\newcommand{\ot}{\leftarrow}
\newcommand{\dembd}{\overset{\text{d}}{\hookrightarrow}}

\newcommand{\vu}{\vec{u}}
\newcommand{\vv}{\vec{v}}
\newcommand{\vw}{\vec{w}}
\newcommand{\vf}{\vec{f}}

\newcommand{\vvh}{\vec{v}_h}
\newcommand{\vwh}{\vec{w}_h}

\newcommand{\dt}[1]{\dot #1} 
\newcommand{\dtt}[1]{\ddot #1} 
\newcommand{\ddx}{\d x}

\newcommand{\OOmega}{{\overline{\Omega}}}

\newcommand{\intB}{\int_{\Omega}}
\newcommand{\intS}{\int_{\Gamma}}

\DeclareMathOperator{\Div}{div}
\newcommand{\grad}{\nabla}
\newcommand{\gradS}{\nabla_\Gamma}

\newcommand{\divS}{\Div_\Gamma}
\newcommand{\LaplS}{\Delta_\Gamma}

\newcommand{\wdampB}{\alpha_\Omega}
\newcommand{\wdampS}{\alpha_\Gamma}
\newcommand{\advB}{\mathbf{v}_\Omega}
\newcommand{\advS}{\mathbf{v}_\Gamma}

\newcommand{\diffB}{c_\Omega}

\newcommand{\diffS}{c_\Gamma}

\newcommand{\oscB}{a_\Omega}
\newcommand{\oscS}{a_\Gamma}

\newcommand{\dynS}{\mu}

\newcommand{\sosrcB}{\sosrc_\Omega}
\newcommand{\sosrcS}{\sosrc_\Gamma}

\newcommand{\advBh}{\mathbf{v}_{\Omega_h}}
\newcommand{\advSh}{\mathbf{v}_{\Ga_h}}

\DeclarePairedDelimiterXPP{\ipGen}[3]{}{\big\lparen}{\big\rparen}{_{#3}}{#1\mid #2}
\DeclarePairedDelimiterXPP{\dpGen}[3]{}{\langle}{\rangle}{_{#3}}{#1,#2}

\DeclarePairedDelimiterXPP{\normbbV}[1]{}{\lVert}{\rVert}{_{\bbV}}{#1}
\DeclarePairedDelimiterXPP{\normHbb}[1]{}{\lVert}{\rVert}{_{\bbH^{0}}}{#1}
\DeclarePairedDelimiterXPP{\normHbbmo}[1]{}{\lVert}{\rVert}{_{\bbH^{-1}}}{#1}
\DeclarePairedDelimiterXPP{\normHbbo}[1]{}{\lVert}{\rVert}{_{\bbH^{1}}}{#1}

\newcommand{\HoB}{{H^1(\Omega)}}

\newcommand{\HoS}{H^1(\Gamma)}

\newcommand{\HBS}[1]{H^{#1}(\Omega;\Gamma)}
\DeclarePairedDelimiterXPP{\normHBS}[2]{}{\lVert}{\rVert}{_{\HBS{#2}}}{#1}

\newcommand{\LtB}{{L^2(\Omega)}}
\newcommand{\LtS}{{L^2(\Gamma)}}

\newcommand{\massS}{\mu_\Gamma}
\newcommand{\stiffS}{k_\Gamma}

\newcommand{\ipHasymb}{\overline m}
\DeclarePairedDelimiterX{\ipHa}[2]{\ipHasymb\big\lparen}{\big\rparen}{#1, #2} 
\newcommand{\ipVasymb}{\overline a}
\DeclarePairedDelimiterX{\ipVa}[2]{\ipVasymb\big\lparen}{\big\rparen}{#1, #2} 

\newcommand{\STdomB}{\text{in } \Omega}
\newcommand{\STdomS}{\text{on } \Gamma}

\newcommand{\soiv}{u_0} 
\newcommand{\sois}{u_1} 
\newcommand{\sosrc}{f} 
\newcommand{\so}{\sor} 
\newcommand{\soinv}{\sor^{-1}} 
\newcommand{\sor}{A} 
\newcommand{\soB}{\soBr} 
\newcommand{\soBr}{B} 

\newcommand{\soivh}{\intpo \soiv} 
\newcommand{\soish}{\intpo \sois} 
\newcommand{\sods}{V_h} 
\newcommand{\sosrch}{\intpo \sosrc}
\newcommand{\sorh}{\sor_h} 
\newcommand{\sorhinv}{\sor_h^{-1}} 
\newcommand{\soBrh}{\soBr_h} 

\newcommand{\defect}{d_h}

\newcommand{\los}{\mathcal{L}}

\newcommand{\cqmb}{\rho}
\newcommand{\cqmbh}{\widehat{\rho}}

\newcommand{\ccoerc}{\alpha}

\newcommand{\BC}{boundary conditions}
\newcommand{\dynBC}{dynamic boundary condition}

\newcommand{\sdamp}{strongly damped}
\newcommand{\Sdamp}{Strongly damped}

\newcommand{\eg}{e.g\xperiod}
\newcommand{\ie}{i.e\xperiod}
\newcommand{\cf}{cf\xperiod}

\newcommand{\bbH}{\mathbb{H}}

\newcommand{\bbV}{\mathbb V}
\newcommand{\Hh}{H_h}
\newcommand{\uh}{u_h}
\newcommand{\vh}{v_h}
\newcommand{\wh}{w_h}

\newcommand{\Landau}{\mathcal O}
\newcommand{\eps}{\varepsilon}

\newcommand\bfb{{\mathbf b}}

\newcommand\bfd{{\mathbf d}}
\newcommand\bfe{{\mathbf e}}

\newcommand\bfr{{\mathbf r}}
\newcommand\bfu{{\mathbf u}}
\newcommand\bfv{{\mathbf v}}
\newcommand\bfw{{\mathbf w}}

\newcommand\bfy{{\mathbf y}}

\newcommand\bfA{{\mathbf A}}
\newcommand\bfB{{\mathbf B}}

\newcommand\bfD{{\mathbf D}}
\newcommand\bfE{{\mathbf E}}

\newcommand\bfH{{\mathbf H}}
\newcommand\bfHH{\widehat{\mathbf H}}

\newcommand\bfJ{{\mathbf J}}

\newcommand\bfM{{\mathbf M}}

\newcommand\bfY{{\mathbf Y}}

\newcommand\calA{{\cal A}}
\newcommand\calD{{\cal D}}

\newcommand\calT{{\mathcal T}}

\def\eps{\varepsilon}

\def\AC{\calA}

\newcommand{\andquad}{\qquad \textrm{ and } \qquad}

\def \d {\mathrm{d}}

\renewcommand{\div}{\textnormal{div}\hspace{2pt}}
\newcommand{\Ga}{\Gamma}

\newcommand{\Id}{\textrm{Id}}
\newcommand{\inv}{^{-1}}

\newcommand{\laplace}{\Delta}

\newcommand{\nb}{\nabla}
\newcommand{\Om}{\Omega}

\newcommand{\pa}{\partial}
\newcommand{\R}{\mathbb{R}}

\newcommand{\st}{such that}
\def \t {(t) }
\def \to {\rightarrow}

\newcommand\quadand{\quad\hbox{ and }\quad}

\newcommand\quadfor{\quad\hbox{ for }\quad}
\newcommand\quadfora{\quad\hbox{ for all }\quad}
\newcommand\quadfore{\quad\hbox{ for every }\quad}
\newcommand{\ccdot}{(\cdot,\cdot)}
\newcommand{\ee}{\textnormal{e}}
\newcommand{\diff}{\tfrac{\d}{\d t}}
\newcommand{\half}{{\tfrac{1}{2}}}
\newcommand{\Half}{{\frac{1}{2}}}

\newcommand{\ga}{\gamma}
\newcommand{\nbg}{\nb_{\Ga}}
\newcommand{\nbgh}{\nb_{\Ga_h}}

\def\bbk{\color{black}}
\def\ebk{\color{black}}

\def\bdh{\color{black}}
\def\edh{\color{black}}

\newcommand{\TODO}[1]{%
	\smallskip {\color{red} \hrule \vspace*{1mm}
		\textbf{\underline{To-Do:}} #1 \vspace*{1mm}
		\hrule} \vspace*{1mm}
}

\WithSuffix\newcommand\TODO*[1]{{\color{red}\textbf{\underline{To-Do:}} #1 \textbf{\underline{:oD-oT}}}}

\renewcommand{\nu}{\textnormal{n}}

\begin{document}

\title{Finite element error analysis of wave equations with \dynBC s: $L^2$ estimates}
\shorttitle{$L^2$ norm error analysis of wave equations with \dynBC s}

\author{%
	{\sc
		David Hipp\thanks{Email: david.hipp@kit.edu},} \\[2pt]
	Institute for Applied and Numerical Mathematics, Karlsruhe Institute of Technology \\
	Englerstr. 2, 76131 Karlsruhe, Germany
	\\[6pt]
	{\sc and}\\[6pt]
	{\sc Bal\'azs Kov\'acs}\thanks{Email: kovacs@na.uni-tuebingen.de}\\[2pt]
	Mathematisches Institut, Universit\"at T\"{u}bingen,\\
	Auf der Morgenstelle 10, 72076 T\"{u}bingen, Germany
}
\shortauthorlist{D.~Hipp and B.~Kov\'acs}

\maketitle

\begin{abstract}
	{$L^2$ norm error estimates of semi- and full discretisations, using bulk--surface finite elements and Runge--Kutta methods, of wave equations with dynamic boundary conditions are studied. The analysis resides on an abstract formulation and error estimates, via energy techniques, within this abstract setting. Four prototypical linear wave equations with dynamic boundary conditions are analysed which fit into the abstract framework. For problems with velocity terms, or with acoustic boundary conditions we prove surprising results: for such problems the spatial convergence order is shown to be less than two. These can also be observed in the presented numerical experiments.}
	{wave equations, dynamic boundary conditions, abstract error analysis, Ritz map, $L^2$ error estimates, Runge--Kutta methods.}
\end{abstract}

\section{Introduction}

In this paper we study the $L^2$ error of semi- and full discretisations of wave equations with dynamic boundary conditions using bulk--surface finite elements and Gauss--Runge--Kutta methods. 

Dynamic boundary conditions can account for the momentum of the wave on the boundary and, in particular, for tangential wave propagation along the boundary.
As tangential wave propagation is inherently modelled on (piecewise) smooth boundaries, triangulations of the domains are possibly not exact. 
Therefore, finite element discretisations can become non-conforming which makes the error analysis more involved.
This paper considers four prototypical examples for the class of linear wave-type problems with dynamic boundary conditions:
a simple model problem with only second-order terms, problems with advective terms, problems with strong damping, and problems with acoustic boundary conditions. Albeit stating our main results for these four examples, the main part of our error analysis is done in an abstract setting, and can thus be applied to all linear second order wave equations fitting into this setting.

The modelling and analysis of wave equations with dynamic boundary conditions is an intensively researched field.
Initially, dynamic boundary conditions for wave equations appeared in models of vibrating elastic rods or beams with tip masses attached at their free ends, cf.~\cite{AndKS96} and references therein. 
However, a first derivation of dynamic boundary conditions, as considered in this paper, was given in \cite{Gol06} which also comments on their relation to Wentzell boundary conditions.
This seminal work was recently complemented by \cite{FigR15} which presents a systematic approach to derive (dynamic) boundary conditions for conservative systems via a Lagrangian framework.
Moreover, the analysis of such problems is quite developed.
Let us mention here  \cite{Vit13} and \cite{Vit15} where well-posedness of wave equations with (non-linear) dynamic boundary conditions is shown, \cite{GraL14} which studies the regularity of problems with strong boundary damping, and \cite{GalT17} which proves the Carleman inequality.
Another important category are acoustic boundary conditions which arise in models for wave--structure interactions.
First proposed in \cite{BeaR74}, they continue to be a topic of intensive mathematical and physical research, see for example \cite{GalGG04}, \cite{Mug06}, or \cite{FroMV11} for a non-linear version, as well as \cite{VedGC-DYWvH16}.

Despite the long history of wave equations and, more general, partial differential equations (PDEs) with dynamic boundary conditions, the error analysis of their numerical approximations has mainly been developed during the last few years.
\cite{ElliottRanner} was the first paper to address the non-conformity of finite element approximations for bulk--surface PDEs in curved domains.
It proposes and analyses an isoparametric bulk--surface finite element method for an elliptic coupled bulk--surface problem. 
Finite elements (and non-uniform rational B-splines for the approximation of curved domains) for elliptic problems with dynamic boundary conditions have been analysed in \cite{KasCDQ15}. 
Although \cite{Fairweather} already gave error estimates for (conforming) Galerkin methods for linear parabolic problems, it went unnoticed in the dynamic boundary conditions community, possibly due to the fact that the term \emph{dynamic} has not appeared at all in his paper.
We refer to \cite{dynbc} for a more complete numerical analysis of parabolic problems with dynamic boundary conditions, including surface differential operators, semi-linear problems and time integration.
As for hyperbolic equations, 
\cite{LesZ15} studies the numerical approximation of the special case of wave equations in two asymmetric half-spaces divided by a ``wavy'' surface.
The first convergence estimates for general wave equations with dynamic boundary conditions and isoparametric finite element discretisations thereof were shown in \cite{Hip17}.
These energy norm ($H^1$) estimates are derived using the unified theory for (possibly) non-conforming semi-discretisations of wave-type equations presented in \cite{HipHS17}.
Apart from these two works, we are not aware of papers studying the numerical errors for \emph{general} wave equations with dynamic boundary conditions. 

In this paper we present $L^2$ convergence rates for finite element approximations and for full discretisations with Gauss--Runge--Kutta methods of wave equations with dynamic boundary conditions by combining the ideas of \cite{dynbc}  and \cite{HipHS17}, with those of \cite{Mansour_GRK,HP_RK}.
Our approach is based on energy techniques and an abstract formulation of second-order wave equations and their spatial semi-discretisations. 
It can be outlined as follows: 
Via energy estimates, we first reproduce a stability estimate in a weak norm for the continuous problem from  \cite{Hip17}. 
These weak norm estimates entail a $L^2$ norm stability result.
For the $L^2$ error analysis, we therefore derive an analogous stability estimate for the abstract semi-discrete problem in discrete weak norms.
Then, using the abstract Ritz map from \cite{dynbc}, we show an error estimate in terms of errors in the initial value and the semi-discrete defect, and further prove that the latter is bounded by geometric (in the abstract setting conformity, cf.~\cite{Hip17}) and approximation (i.e.~interpolation and Ritz map) error estimates. 
Up to this point, the analysis does not use any specific information on the particular terms of the bilinear forms and, in particular, the boundary conditions. 
Finally, we obtain  $L^2$ convergence rates for each example separately, by studying the different error terms, using properties of the wave equation and (geometric, interpolation and Ritz map) approximation results for the bulk--surface finite element method.

The geometric approximation errors for the terms involving the velocity, do not allow optimal-order convergence rates for all cases. In Table~\ref{tab:overview} we collect the obtained error estimates for the spatial semi-discretisation. We also marked whether these results are illustrated by numerical experiments.
\begin{table}
	\centering
	\setlength{\tabcolsep}{6pt}
	\setlength{\extrarowheight}{3pt}
	\begin{tabular}{llcc}
		Boundary condition  & $L^2$ error & Discussed in & Illustrated \\
		\toprule        %
		purely second-order  & \quad $h^2$  & Theorem~\ref{theorem: semi-discrete error bound: pure second} & \cmark %
		\\
		\midrule
		advection  & \quad $h^{3/2}$  & Theorem~\ref{theorem: semi-discrete error bound: advective}  & \cmark \\
		advection \emph{only} on the surface \qquad & \quad $h^{2}$  & Theorem~\ref{theorem: semi-discrete error bound: advective}  & \cmark \\
		\midrule
		strong damping & \quad $h$ & Theorem~\ref{theorem: semi-discrete error bound: strong damping}  & \xmark \\
		\midrule
		acoustic &  \quad $h^{3/2}$ & Theorem~\ref{theorem: semi-discrete error bound: acoustic boundary conditions}  & \cmark \\
		\bottomrule        
	\end{tabular}
	\caption{Overview of $L^2$ convergence rates for linear finite elements with mesh width $h$ shown in this article.} \label{tab:overview}
\end{table}


We expect that further interesting problems, such as wave equations with new types of dynamic boundary conditions, or with time- and space-dependent coefficients, as well as semi-linear problems can be treated within this setting (or slight modifications of it) using the presented techniques, subject to the error analysis of the mentioned geometric and approximation errors.

Since the matrix--vector formulation of these second order problems coincides with the ODE system for wave problems with standard boundary conditions, the convergence proofs for the full discretisation are straightforward. Some parts have already been covered in the literature, only the $L^2$ norm requires some simple modifications. We give these details, but for those parts which are not new we only give detailed references, following \cite{Mansour_GRK,HP_RK,HPS,rkkato}. We strongly believe, and the previous references also strengthen, that these techniques extend to time discretisations of more general, e.g.~semi- or quasi-linear, problems.


\medskip
\noindent \textit{Outline.} \ In Section~\ref{section:variational formulation for wave eqns} we first introduce the abstract framework, its assumptions, norms and bilinear forms.
\bdh The main motivation of this paper are the four exemplary wave equations with dynamic boundary conditions presented in Section~\ref{section:wave eqn with dynbc}.
\edh There, we also show how their respective variational formulation fits into the abstract framework by giving the suitable Hilbert space \bdh and bilinear forms and we state sufficient conditions on the coefficients for well-posedness. \edh

Section~\ref{section:finite element method} starts with a description of the bulk--surface finite element method and the strategy for dealing with the approximation of a smooth domain with possibly curved boundary. \bdh In Section~\ref{subsection:semi discrete abstract wave equations}, we then define the abstract framework for semi-discretizations of wave equations and state the central error estimate. \edh

Section~\ref{section:main results} presents the main results of this paper.
\bdh
For each example from Section~\ref{section:variational formulation for wave eqns}, we give the bilinear forms discretized by the finite element method and the semi-discrete error estimates in the $L^2$ norm.

To prove these error estimates, we proceed in two steps.
First, Section~\ref{section:abstract error analysis} contains the error analysis in the abstract framework. There we show continuous and semi-discrete stability estimates in a weak norm by using energy techniques. Our main abstract result is an error estimate in terms of several approximation errors. To show this the semi-discrete stability bound is applied to the error equation, then the appearing defect is shown to be bounded by approximation and geometric errors (i.e.~errors in the interpolation and the Ritz map, and errors due to the semi-discrete bilinear forms).
Second, in Section~\ref{section:FEM error analysis}, we prove the error estimates for these approximation errors by using interpolation and geometric error estimates available for the FEM. This section is split into four parts, each devoted to one the exemplary wave equations with dynamic boundary conditions and its FEM discretization, and giving the proof of the corresponding theorem in Section~\ref{section:main results}. 
\edh

In Section~\ref{section:time discreteisations} we turn to time discretisations \bdh and show how stability estimates and convergence results for Gauss--Runge--Kutta methods are shown for the studied abstract wave equations, i.e.~for the four problems with dynamic boundary conditions. \edh The required modifications, compared to the literature, are presented in detail.

In Section~\ref{section:numerics} we present various numerical experiments -- to all four problems -- illustrating our theoretical results.
\bdh
In particular, we show that the proven fractional convergence rates for finite element discretizations of wave equations with advective and acoustic boundary conditions can be observed in numerical experiments, \edh cf.~the last column of Table~\ref{tab:overview}.



\bdh In order to help our readers only interested in the abstract setting and error analysis, or those only in the error analysis of a particular wave equation with dynamic boundary condition, the corresponding parts of the paper are shown in Table~\ref{tab:table of contents}.  \edh

\begin{table}[h]
  \centering
  \setlength{\tabcolsep}{6pt}
  \setlength{\extrarowheight}{3pt}
  \begin{tabular}{r c l}
    \textbf{Wave equations with dynamic b.c.}  & \rotatebox[origin=c]{-90}{\small{Section\; }} & \textbf{Abstract second-order wave equations}\\
    \toprule
    & \ref{section:abstract framework} & abstract problem and well-posedness 
    \\              
    exemplary PDEs with variational form & \ref{section:wave eqn with dynbc} &  \\ 
    the bulk-surface FEM & \ref{subsection:bulk surface finite element method} & \\ 
    & \ref{subsection:semi discrete abstract wave equations} & semi-discrete problem and error estimate \\ 
    convergence results & \ref{section:main results} & \\
    & \ref{section:abstract error analysis} & error analysis in weak norm  \\ 
    analysis of FEM approximations & \ref{section:FEM error analysis} & application of abstract error estimate \\ 
    & \ref{section:time discreteisations} & full discretization including error analysis \\ 
    numerical experiments & \ref{section:numerics} &  \\
  \end{tabular}
  \caption{Categorized table of contents.}
  \label{tab:table of contents}
\end{table}

\section{Analysis of wave equations with dynamic boundary conditions}
\label{section:variational formulation for wave eqns}

In this section, we present an abstract setting for wave equations, similar to the ones in \cite{dynbc} and \cite{HipHS17}, and then consider different examples of wave equations with dynamic boundary conditions fitting into this abstract framework. 

\subsection{Abstract framework}
\label{section:abstract framework}

Let $V$ and $H$ two real Hilbert spaces with norms $\normV{\cdot}$ and  $\normH{\cdot}$, the latter norm induced by the inner product $m(\cdot,\cdot)$ on $H$, such that $V$ is densely and continuously embedded in $H$ (i.e.\ $\normH{u} \leq c \normV{u}$). 
Furthermore, we identify $H$ with its dual $H^*$ which defines the Gelfand triple
\begin{align*}
V \bbk \dembd \ebk H \simeq H^* \bbk \dembd \ebk V^*.
\end{align*}

As a consequence of this identification, the duality $\dpV{\cdot}{\cdot} \colon V^* \times V \to \R$ coincides with $m(\cdot,\cdot)$ on $H \times V$. 

The general abstract wave equation, which covers all examples in this paper, reads: 
Find $u \colon [0,T] \to V$ \st 
\begin{subequations}
	\label{eq:sowt-var}
	\begin{align}
	\dpV{\ddot u(t)}{v} + \sobfB{\dot u(t)}{v} + \sobf{u(t)}{v}  &= \dpV{\sosrc(t)}{v} , \qquad \forall v \in V, \label{eq:sowt-var-evol}\\
	u(0) &= \soiv, \quad \dot u(0) = \sois,   \label{eq:sowt-var-ivs}
	\end{align}
\end{subequations}
where $\sosrc \colon [0,T] \to V^*$ is a given function, \bbk $u_0, u_1 \in V$ are given initial data, \ebk and where $a \colon V \times V \to \R$ and \mbox{$\sobfBsymb \colon V \times V \to \R$} are continuous bilinear forms such that $\sobfBsymb + \cqmb \ipHsymb$ is monotone for some $\cqmb \geq 0$, \ie, 
\begin{equation}
\label{eq:quasi monotonicity of b}
\sobfB{v}{v} +\cqmb \normH{v}^2  \geq 0 \quadfore  v \in V,
\end{equation}
and $\sobfsymb$ is \bbk symmetric, and \ebk coercive with an $\ccoerc > 0$:
\begin{equation}
\label{eq:sobf-coerc}
\sobf{v}{v} \geq \ccoerc \normV{v}^2 \quadfore  v \in V .  
\end{equation}

The above variational equation \eqref{eq:sowt-var} can be written as the evolution equation in $V^*$
\begin{equation}
\label{eq:abstract evolution eq}
\ddot u \t + \soB \dot u \t + \so u \t = \sosrc \t ,
\end{equation}
where $\so, \soB \in \los (V,V^*)$ are induced by the bilinear forms $\sobfsymb$ and $\sobfBsymb$ via
\begin{equation}
\label{eq:operators A and B def}
\dpV{\so w}{v} = \sobf{w}{v}, \andquad \dpV{\soB w}{v} = \sobfB{w}{v}, \qquad w,v \in V. 
\end{equation}

Note that, due to our assumptions, $\so$ is an isomorphism by the Lax--Milgram theorem and $\sobfsymb$ is an inner product on $V$ such that
\begin{align*}
\normVa{v}^2  \defeq \sobf{v}{v}
\end{align*}
defines an equivalent norm, satisfying
\begin{equation*}
\sqrt{\ccoerc} \normV{v} \leq \normVa{v} \leq \norm{\so}_{\Vdual \ot V}^{1/2} \normV{v}, \qquad v \in V.
\end{equation*}
\bbk From now on, on $V$ we will almost exclusively use the $a$ induced norm $\normVa{\cdot}$. \ebk 

The abstract wave equation \eqref{eq:sowt-var} is well-posed in different settings.
The following theorem collects a weak and a strong well-posedness result which are shown using semigroup theory and, for the weak result, the theory of Sobolev towers.
For the proof, we refer to \cite[Theorem~4.3~and~4.13]{Hip17} and note that the strong result is shown in \cite{Sho10}.


\begin{theorem}\label{thm:w-p}  
	Let the above assumptions be fulfilled and let the initial values \mbox{$\soiv \in V$}, $\sois \in H$ and source term satisfy
	$\sosrc \in C^1([0,T];V^*) + C([0,T];H)$.
	Then there exists a unique solution $u$ of \eqref{eq:sowt-var} \st
	\begin{align}
	\label{eq:weak regularity}
	u \in C^1([0,T];H) \cap C([0,T]; V) \quadand \dot u +  \soB u \in C^1([0,T]; V^* ) .
	\end{align}
	
	If furthermore $\soiv,\sois \in V$ \st\ $\so \soiv + \soB \sois \in H$ and
	$\sosrc \in C^1([0,T];H)$ 
	or $\col{\sosrc , \soB \sosrc}  \in C([0,T];V \times H)$,  
	then there exists a unique solution $u$ of \eqref{eq:sowt-var} \st
	\begin{align}
	\label{eq:strong regularity}
	u \in C^2([0,T];H) \cap C^1([0,T];V) \quadand \so u + \soB \dot u  \in C([0,T];H) .
	\end{align}
\end{theorem}
\bbk We assume that the inhomogeneity $f$ and the initial values $u_0, u_1$ satisfy the above conditions. \ebk

\subsection{Wave equations with dynamic boundary conditions}
\label{section:wave eqn with dynbc}

While the error analysis provided in this paper is done for the abstract wave equation \eqref{eq:sowt-var}, we will discuss the numerical solution and the convergence behaviour of four exemplary wave equations with \emph{dynamic boundary conditions} in detail.
In the following, we will introduce these examples and show how the corresponding variational formulations can be written as an abstract wave equation of the form of \eqref{eq:sowt-var}. \bbk For proving that the abstract assumptions of Section~\ref{section:abstract framework} are satisfied by these problems we refer to \cite[Chapter~6]{Hip17}, giving the precise locations therein below. \ebk


Let us briefly introduce some notations.
Let the bulk $\Om \subset\R^d$ ($d=2$ or $3$) be a bounded domain, with (at least) $C^2$ boundary $\Gamma = \partial\Omega$, which is referred to as the surface. Further, let $\nu$ denote the unit outward normal vector to $\Ga$.
Then the surface gradient on $\Ga$, of a function $u : \Ga \to \R$, is denoted by $\nbg u$, and is given by $\nbg u = \nb \bar u -( \nb \bar u \cdot \nu) \nu$, while the Laplace--Beltrami operator on $\Ga$ is given by $\laplace_\Ga u = \nbg \cdot \nbg u$.
Moreover, $\ga u$ denotes the trace of $u$ on $\Ga$, and $\pa_{\nu} u$ denotes the normal derivative of $u$ on $\Ga$.
Finally, temporal derivatives are denoted by $\dot{\phantom{u}} = \d / \d t$.

\subsubsection{Purely second-order dynamic boundary conditions}
\label{section:purely second-order wave equations}

For mathematical models of wave phenomena, the main region of interest is (often) given by the volume of the transmission medium that propagates the waves.
This transmission medium therefore defines the domain for the wave equation while boundary conditions are used to effectively model the behaviour of the wave at the border to its surrounding.
If these effective models capture oscillations of the surrounding structure or waves propagating along its surface, then we call it a dynamic boundary condition.

Here we consider the prototype of such a situation where the boundary condition is another wave equation on which the normal derivative of the bulk function acts as a force.
Depending on the authors, such boundary conditions have been called oscillatory or kinetic, cf.~\cite{GalT17} or  \cite{Vit13}. 
We begin with an example of a wave equation endowed with dynamic boundary conditions which only contains second-order terms modelling local oscillations and propagation of waves along the boundary:
Find the solution \mbox{$u \colon [0,T] \times \OOmega \to \R$ }
\begin{equation} 
\label{wave eqn - pure second-order}
\begin{alignedat}{4}
\ddot u &= \laplace u + f_\Om &\qquad & \textrm{in} \quad \Om , \\
\mu \ddot u &= \beta \laplace_\Ga u - \kappa u  - \pa_{\nu} u + f_\Ga &\qquad & \textrm{on} \quad  \Gamma,
\end{alignedat}
\end{equation}
where the constants $\mu$ and $\kappa$ are positive, $\beta$ is non-negative and  $f_\Om:[0,T] \times\Om \to \R$ and $f_\Ga : [0,T] \times\Ga \to \R$ are given functions. Here we do not consider problems with tip masses, i.e.~where $\kappa= \beta = 0$, see e.g.~\cite{AndKS96}, however we expect them to be treatable with our techniques, although with more technicalities. 

The variational formulation of \eqref{wave eqn - pure second-order} can be cast as the abstract wave equation \eqref{eq:sowt-var} in the Hilbert spaces
\begin{equation}
\label{eq:Hilbert spaces}
\begin{aligned}
V = &\ \{ v \in H^1(\Om) \mid \beta \ga u \in H^1(\Ga) \} \andquad \\
H = &\ L^2(\Om) \times L^2(\Ga) ,
\end{aligned} \quad \textnormal{with the embedding } \quad v \mapsto (v, \ga v) ,
\end{equation}
see \cite[Corollary~6.7]{Hip17}. \bbk For brevity we will abbreviate the pairs $(v, \ga v)$ in $H$ by their first component $v$. \ebk 

The inner products on $H$ and $V$ are given by
\begin{equation}
\label{eq:bilinear forms}
\begin{aligned}
\bbk \ipH{(w,\omega)}{(v,\psi)} = \ebk &\ \int_\Om w v \ \d x + \mu \int_\Ga \bbk \omega \psi \ebk \ \d \sigma \andquad \\
a(w,v) = &\ \int_\Om \nb w \cdot \nb v \ \d x + \beta \int_\Ga \nbg w \cdot \nbg v \ \d \sigma + \kappa \int_\Ga (\ga w) (\ga v) \ \d \sigma ,
\end{aligned}
\end{equation}
where for brevity we write $\nbg v$ instead of $\nbg (\ga v)$. \bbk According to the notational convention above, for embedded pairs, we will often write $m(w,v)=m((w,\ga w),(v,\ga v))$. \ebk We will employ these notations throughout the paper.

Furthermore, since there is no velocity term, we have $\sobfBsymb = 0$ and the right-hand side function $f$ is understood as
\begin{align*}
\dpV{f}{v} = \int_\Om f_\Om v \ \d x + \int_\Ga f_\Ga (\ga v) \ \d \sigma .
\end{align*}

\bbk In \cite[Lemma~6.3 and Section~6.2.2]{Hip17} it is shown that the abstract assumptions from above are satisfied. \ebk 

\subsubsection{Advective dynamic boundary conditions}

Waves propagating through a medium in motion are subject to advection effects, which lead to terms containing first-order time derivatives, cf.~\cite[wave eq. W4]{Cam07}. 
The following example accounts for advective and (weak) damping effects in the bulk and on the surface:
We seek the solution $u \colon [0,T] \times \OOmega \to \R$ of 
\begin{subequations}
	\label{eq:wave eqn - advective}
	\begin{alignat}{4}
	\dtt u = &\ \laplace u - \big ( \wdampB + \advB \cdot \grad \big )  \dt u + \sosrcB\qquad && \STdomB , \\
	\dynS \dtt u = &\ \beta \LaplS u - \big ( \wdampS + \advS \cdot \gradS \big ) \dt u  - \kappa u - \pa_{\nu} u + \sosrcS \qquad && \STdomS . 
	\end{alignat}
\end{subequations}
Here $\dynS,\beta,\kappa > 0$, $ \wdampB,\wdampS \geq 0$ are constants and $\advB \in L^\infty(\Omega; \R^d)$, $\advS \in L^\infty(\Ga; \R^d)$ are given vector fields with $ \Div  \advB  \in L^\infty(\Om)$, $\divS  \advS \in L^\infty(\Ga)$ such that 
\begin{equation}
\label{eq:advective eq coefficient assumptions}
\wdampB -\frac 12  \Div  \advB  \geq 0 \quad \text{ in } \Omega \quadand \wdampS + \frac 12 \big( \advB \cdot \nu - \divS  \advS \big ) \geq 0 \quad \text{ on } \Gamma .
\end{equation}
These last assumptions guarantee that $\sobfBsymb$ is monotone, \cf \cite[Lemma~6.3 and Section~6.2.2]{Hip17}. 
Note that for undamped models, i.e.~if $\wdampB = \wdampS = 0$, the first condition implies that $\advB$ has no sources and the second one that any flow of $\advB$ over $\Gamma$ is compensated by $\advS$.

This problem can also be written in the abstract form \eqref{eq:sowt-var}, using the same Hilbert spaces $H$ and $V$ from \eqref{eq:Hilbert spaces}, and the duality and bilinear form from \eqref{eq:bilinear forms}, while $b$ is now given by
\begin{align}
\label{eq:bilinear forms - advective}
\sobfB{w}{v} & \defeq \intB \big ( \wdampB w  +  \advB \cdot \grad  w \big) v \ \ddx + \intS \big ( \wdampS  \ga w +  \advS \cdot \gradS  w \big) \ga v \ \d\sigma .
\end{align}

\subsubsection{\Sdamp\ dynamic boundary conditions}

Strong damping is of great relevance in engineering due as it increases the robustness of systems against perturbations.
Boundary conditions involving strong damping are particularly interesting for applications \bbk for physical phenomena that exhibit both elasticity and viscosity when undergoing deformation and \ebk for wave--structure interactions, see~\bbk \cite{GraberShomberg}, \ebk \cite{GraL14} and \cite{Nic17}.

We seek $u \colon [0,T] \times \OOmega \to \R$  such that
\begin{equation}
\label{eq:wave eqn - strong damping}
\begin{alignedat}{4}
\ddot u &= d_\Om \laplace \dot u + \laplace u + f_\Om &\qquad & \textrm{in} \quad \Om , \\
\mu \ddot u &= d_\Ga \laplace_\Ga \dot u + \beta \laplace_\Ga u - \kappa u  - \pa_{\nu} u - d_\Om \pa_{\nu} \dot u + f_\Ga &\qquad & \textrm{on} \quad  \Gamma,
\end{alignedat}
\end{equation}
with the same constants as in \eqref{wave eqn - pure second-order}, except again $\beta$ is assumed to be positive, and additionally with the damping coefficients $d_\Om,d_\Ga > 0$.

\smallskip
The weak formulation of this problem again fits into the framework of \eqref{eq:sowt-var}, by using the same spaces as before \eqref{eq:Hilbert spaces}, and using the duality and bilinear form defined in \eqref{eq:bilinear forms}, and 
\begin{equation}
\label{eq:bilinear forms - strong damping}
b(w,v) = d_\Om \int_\Om \nb w \cdot \nb v \ \d x + d_\Ga \int_\Ga \nbg w \cdot \nbg v \ \d \sigma .
\end{equation}
Note that we have to apply Green's formula in the bulk twice and then insert the boundary condition for $\pa_{\nu} u + d_\Om \pa_{\nu} \dot u$ to derive the variational formulation, \bbk cf.~Section~6.1 and Lemma~6.3 in \cite{Hip17}. \ebk

\subsubsection{Acoustic boundary conditions}
\label{section:acoustic boundary conditions}

The wave equation with {acoustic boundary condition} models the propagation of sound waves in a fluid at rest filling a tank $\Omega$, whose walls $\Gamma$, are subject to small oscillations in normal direction and elastic effects in tangential direction. 
The model is described by the following system: Seek the acoustic velocity potential $u \colon [0,T] \times \OOmega \to \R$ and the displacement of $\Gamma$ in normal direction $\delta \colon [0,T]\times \Gamma \to \R$ such that
\begin{subequations}
	\label{eqs:ex-acoustic-pde}
	\begin{alignat}{3} 
	\dtt u   &= - \oscB u + \diffB \Delta u + \sosrcB & \qquad &  \STdomB, \label{eqs:ex-acoustic-pde-1}\\
	\massS \dtt \delta  &= - \stiffS \delta +  \diffS \Delta_\Gamma \delta -  \diffB \dt u + \sosrcS  & \qquad & \STdomS, \label{eqs:ex-acoustic-pde-2}\\
	\dt \delta &= \pa_\nu u & \qquad &  \STdomS \label{eqs:ex-acoustic-pde-3},
	\end{alignat}    
\end{subequations}

where we assume that $\diffS, \diffB, \massS, \oscB, \stiffS > 0$ are constants.
This model was first proposed in \cite{BeaR74} and its analytical properties continue to be a topic of research.
See, \eg, \cite{GalGG04} for a comparison with Wentzell boundary conditions, \cite{Mug06} for a spectral analysis using operator matrices and \cite{FroMV11} for well-posedness analysis of a non-linear version.


For problems with acoustic boundary conditions, we denote functions in the \emph{bulk by Roman letters}, functions on the \emph{surface by Greek letters}, and functions in the \emph{bulk--surface product space} are labelled with $\vec{~}$\ , and usually denoting the vector with the same latter as the bulk function, e.g.\ $\vec w = (w,\omega)$.

The variational formulation of \eqref{eqs:ex-acoustic-pde} is obtained by testing the bulk and surface equations separately by $v \in H^1(\Om)$ and $\psi \in H^1(\Ga)$, using Green's formula on the surface and the bulk, and finally add up the equations, cf.~\cite[Section~6.3]{Hip17}. 
To write this as an abstract wave equation, we use the product spaces 
\bbk 
\begin{equation}
\label{eq:spaces for acoustic bc}
	\begin{aligned}
		V = &\ \HoB \times \HoS \\
		H = &\ \LtB \times \LtS
	\end{aligned}
\end{equation}
\ebk 
and obtain the following problem:
Find $\vec u = (u,\delta) \colon [0,T] \to V$ such that
\begin{equation*}
\dpV{\ddot{\vec u}}{\vec v} + b(\dot{\vec u},\vec v) + a(\vec u,\vec v) = \dpV{\vec f}{\vec v} , \quadfore \vec v \in V ,
\end{equation*}
where the duality and the bilinear forms are given by, for functions $\vw  = \col{ w, \omega}$ and $v  = \col{ v, \psi}$,
\begin{subequations}
	\label{eqs:acoustic-bc-ingredient-var-form}
	\begin{align}
	\ipH{\vw}{\vec v}  \defeq &\ \intB w v \ \ddx + \intS \massS \omega \psi \ \d\sigma,
	, \\
	\label{eq:acoustic bilinear form - b}
	\sobfB{\vw}{ \vec v } \defeq &\  \diffB \intS (\ga w) \psi  - \omega (\ga v) \ \d\sigma,
	, \\
	\label{eq:acoustic bilinear form - a}
	\sobf{\vw}{ \vec v } \defeq &\ \intB \oscB w  v  + \diffB \nabla w \cdot \nabla v \ \ddx  
	+ \intS  \stiffS \omega \psi + \diffS \gradS \omega \cdot  \gradS\psi \ \d\sigma,
	\end{align}
\end{subequations}
and the right hand-side function for acoustic boundary conditions is understood as
\begin{align*}
\dpV{\vf}{\vv} = \int_\Om f_\Om v \ \d x + \int_\Ga f_\Ga  \psi \ \d \sigma .
\end{align*}
Note that $\sobfsymb$ is coercive with $\ccoerc =\min \{ \diffB, \oscS, \diffS, \stiffS \}$ and that $\sobfBsymb$ is skew-symmetric and therefore monotone, \bbk cf.~\cite[Section~6.3]{Hip17}. \ebk

\section{Spatial discretization with the finite element method}
\label{section:finite element method}

For the numerical solution of the above examples we consider a linear finite element method. 
In the following, from \cite{ElliottRanner} and \cite[Section~3.2.1]{dynbc}, we will briefly recall the construction of the discrete domain, the finite element space and the lift operation \bdh which can be used discretize the particular problems of Section~\ref{section:wave eqn with dynbc} in space. Then we will present the abstract framework for spatial discretizations of \eqref{eq:sowt-var} and state the main abstract error estimate.  \edh

\subsection{The bulk-surface finite element method}
\label{subsection:bulk surface finite element method}
The domain $\Om$ is approximated by a triangulation $\calT_h$ with \mbox{maximal mesh width $h$.}
The union of all elements  of $\calT_h$ defines the polyhedral domain $\Om_h$ whose boundary $\Ga_h := \pa \Om_h$ is an interpolation of $\Ga$, i.e.~the vertices of $\Ga_h$ are on $\Ga$.
Analogously, we denote the outer unit normal vector of $\Ga_h$ by $\nu_h$.
We assume that $h$ is sufficiently small to ensure that for every point $x\in\Ga_h$ there is a unique point $p\in\Ga$ such that $x-p$ is orthogonal to the tangent space $T_p\Ga$ of $\Ga$ at $p$.
For convergence results, we consider a quasi-uniform family of such triangulations $\calT_h$ of $\Om_h$.
The finite element space $S_h \nsubseteq H^1(\Om)$ corresponding to $\calT_h$ is spanned by continuous, piecewise linear nodal basis functions on $\Omega_h$, satisfying for each node $(x_k)_{k=1}^N$
$$
\phi_j(x_k) = \delta_{jk}, \quadfor j,k = 1, \dotsc, N .
$$
Then the finite element space is given as
$$
S_h = \textnormal{span}\{\phi_1, \dotsc, \phi_N \} .
$$
We note here that the restrictions of the basis functions to the boundary $\Ga_h$ again form a surface finite element basis over the approximate boundary elements.

Following \cite{Dziuk88}, we define the \emph{lift} of functions $v_h:\Ga_h\to \R$ to $v_h^\ell:\Ga\to\R$ by setting $v_h^\ell(p)=v_h(x)$ for $p\in\Gamma$, where $x\in\Ga_h$ is the \emph{unique} point on $\Ga_h$ with $x-p$  orthogonal to the tangent space $T_p\Ga$.
We further consider the \emph{lift} of functions $v_h:\Om_h\to \R$ to $v_h^\ell:\Om\to\R$ by setting $v_h^\ell(p)=v_h(x)$  if $x\in\Omega_h$ and $p\in \Omega$ are related as \bbk described in detail \ebk in \cite[Section~4]{ElliottRanner}. 
\bbk The mapping $G_h : \Om_h \to \Om$ is defined piecewise by, for an element $E \in \calT_h$,
\begin{equation}
\label{eq:bulk mapping}
	G_h|_E (x) = F_e\big((F_E)^{-1}(x)\big), \qquad \text{for } x \in E,
\end{equation}
where $F_e$ is a $C^1$ map from the reference element onto the smooth element $e \subset \Om$, and $F_E$ is the standard affine liner map between the reference element and $E$.
\ebk 
The \bbk \emph{inverse lift} \ebk $v^{-\ell}:\Ga_h \to \R$ denotes a function whose lift is $v:\Ga \to \R$, and similarly for the bulk as well. 
Note that both definitions of the lift coincide on $\Gamma$. Finally, the lifted finite element space is denoted by $S_h^\ell$, and is given as $S_h^\ell = \{  v_h^\ell \mid v_h \in S_h \}$.

\bdh

\subsection{Semi-discretization of wave equations}
\label{subsection:semi discrete abstract wave equations}

The finite element approximation of a wave equation is based on its variational formulation with integrals over $\Om$ and $\Ga$ replaced by integrals over $\Om_h$ and $\Ga_h$, respectively.
Using the finite element space for a Galerkin ansatz for this variational problem on the polygonal domain $\Om_h$ then yields the semi-discrete problem.

Hence spatial discretisations of wave equations stemming from the finite element method can be written as abstract differential equations in a finite dimensional Hilbert space $V_h \nsubseteq V$ (specified in Section~\ref{section:main results}):
Find the solution $u_h \colon [0,T] \to \sods$ 
such that, for all $\vh \in V_h$,
\begin{subequations} 
	\label{eq:sowt-h-var}
	\begin{align}
	\ipHh{\ddot u_h(t)}{\vh} + \sobfBh{\dot u_h (t)}{\vh} + \sobfh{u_h(t)}{\vh} = &\ \ipHh{\sosrch(t)}{\vh} , \\ 
	u_h(0) = &\ \soivh,  \qquad
	\dot u_h(0) = \ \soish .
	\end{align}  
\end{subequations}
Where for a continuous function $v \in V$, we denote by $\widetilde I_hv \in V_h$ the nodal interpolation of $v$.

First note that the discrete bilinear forms $\ipHhsymb$, $\sobfhsymb$ and $\sobfBhsymb$ inherit the properties from their continuous counterparts. 
In particular, the two norms on $V_h$ 
\begin{equation}
\label{eq:discrete norms}
\begin{aligned}
|v_h|_h^2 = m_h(v_h,v_h)  \quadand
\|v_h\|_h^2 = a_h(v_h,v_h) .
\end{aligned}
\end{equation}
satisfy $|v_h|_h \leq C \| v_h \|_h$ and there exists a constant $\cqmbh \geq 0$ \st\ $\sobfBhsymb + \cqmbh \ipHhsymb$ is monotone. 
 
Since in general $V_h \nsubseteq V$, we can not directly compare the finite element solution $u_h(t)$ at time $0 \leq t \leq T$ with $u(t)$. For the error analysis, we use the  lift operator $\cdot^\ell \colon V_h \to V$ of discrete functions as introduced in Section~\ref{section:finite element method}.  
Due to \cite[Lemma~3.9]{dynbc}, its norm is equivalent to the discrete norm of the function itself:
\begin{equation}
\label{eq:norm equiv}
\|v_h^\ell\| \sim \|v_h\|_h \quadand |v_h^\ell | \sim |v_h |_h
\quad\hbox{uniformly in $h$.} 
\end{equation}
The abstract error analysis presented in the rest of this section applies if the semi-discretization satisfies theses properties.

We define a discrete dual norm on the space $V_h$
\begin{equation}
\label{eq:star and H_h^-1 norm equality}
\normVhdual{d_h} =  \sup_{0\neq v_h \in V_h} \frac{\ipHh{d_h}{\vh}}{\normVh{v_h}}. 
\end{equation}
It is easy to see that, as in continuous case, there exist constants $C,c > 0$ \st 
\begin{align*}
c  \normVhdual{v_h} \leq \normHh{v_h} \leq C \normVh{v_h}
\end{align*}
and that $\normVhdual{\cdot}$ is induced by the inner product $\ipHh{\sorhinv \cdot}{\cdot}$ where the linear operators $\sorh,\soBrh \colon V_h \to V_h$ are given by 
\begin{equation}
\label{eq:A_h defintion}
\sobfh{w_h}{v_h} = \ipHh{\sorh w_h}{ v_h} \quadand     \sobfBh{w_h}{v_h} = \ipHh{\soBrh w_h}{ v_h}.
\end{equation} 

We further introduce the following differences between the continuous and discrete bilinear forms: For any $\wh, \vh \in V_h$, we define
\begin{align*}
\DeltaipH{\wh}{\vh} &= \ipH{\lift \wh}{\lift \vh} - \ipHh{\wh}{\vh}, \\
\DeltasobfB{\wh}{\vh} &= \sobfB{\lift \wh}{\lift \vh} - \sobfBh{\wh}{\vh}.
\end{align*}
Furthermore, we will use the notation 
\begin{align*}
\normVhdual{\DeltaipH{w_h}{\cdot}} = \sup_{v_h \in V_h} \frac{\DeltaipH{w_h}{v_h}}{\normVh{v_h}} .
\end{align*}

For the error analysis, we split the error using the Ritz map $\widetilde R_h u\in V_h$ which for $u\in V$ is defined by
\begin{equation}
\label{Ritz}
	\sobfh{\Ritz u}{\vh} = \sobf{u}{\lift \vh}, \quadfore v_h \in V_h .
\end{equation}
The Ritz map is well-defined for all $u \in V$ due to the abstract assumptions (here, in particular by the coercivity, though $a$ satisfying a G\aa rding inequality suffices with a slight modification), see \cite[Section~3.4]{dynbc}.
Note that, for example, the bilinear form $a$ contains boundary terms which influence $\Ritz$.
Using the notation from \cite[Section~3.4]{dynbc}, we will write $\Ritzl u := (\Ritz u)^\ell \in V_h^\ell$ for the lifted Ritz map. 

In Section~\ref{section:abstract error analysis}, we prove the following error estimate for the lifted solution of the semi-discrete abstract wave equations in the $H$-norm
\begin{subequations}
  \begin{align}
    &\ \normH{u \t - \lift u_h \t} \leq \normH{ u \t - \Ritzl u \t } 
    + C \ee^{\cqmbh T} \bigg ( \eps_0^2 + T \int_0^t d_h(s) \d s \bigg )^{1/2} ,
    \intertext{where $\eps_0$ is the error in the initial values and defined as}
    &\ \begin{aligned}
    \eps_0 = &\ \normH{\intpol \sois - \sois} + \normH{\intpol \soiv - \soiv } + \normH{\Ritzl \sois - \sois} + \normH{\Ritzl \soiv - \soiv} \\
    &\ + \|\Delta b(\Ritz \soiv - \soiv^{-\ell},\cdot)\|_{\star,h} + \|\Delta b(\soivh - \soiv^{-\ell},\cdot)\|_{\star,h} \\
    &\ + c \|B(\Ritzl \soiv - \soiv)\|_{\star} + c \|B(I_h \soiv - \soiv)\|_{\star} ,
    \end{aligned}
    \intertext{the defect $D_h$ can be bounded by}
    &\
    \begin{aligned}
      d_h 
      \leq  C \Big( & \normVadual{\sosrc - \intpol \sosrc} + \normVadual{\Ritzl \ddot u - \ddot u} + \normVadual{\soB (\Ritzl \dot u - \dot u)} \\
      &\ +  \normVhdual{\DeltaipH{\Ritz \ddot u}{\cdot}} + \normVhdual{\DeltasobfB{\Ritz \dot u}{\cdot}}  + \normVhdual{\DeltaipH{\sosrch}{\cdot}}  \Big ) ,
    \end{aligned}
  \end{align}
  and $\intpol v = \lift{(\intpo v)}$ denotes the lifted interpolation of $v$. 
\end{subequations}

In Section~\ref{section:FEM error analysis}, we will apply this error estimate to bulk-surface FEM discretizations of our four examples.
To prove the convergence rate, we then estimate the right-hand side terms which consist of interpolation and geometric errors in problem dependent semi-norms and prove lower bounds for their rate of convergence. 

\edh

\section{$L^2$ error bounds for wave equations with dynamics boundary conditions}
\label{section:main results}



In the rest of the section, we consider the finite element approximation of the examples from Sections~\ref{section:purely second-order wave equations}-\ref{section:acoustic boundary conditions}. 
First, we give concrete definitions for the respective finite element approximation \eqref{eq:sowt-h-var} and then we state the corresponding $L^2$ error estimates with convergence rates for the $L^2$ error of the (lifted) finite element approximation.

\begin{remark}
	The following error estimates require spatial regularity of the solution and its time derivatives.
	Constants in the error estimates will depend on the canonical norm of the space $H^2(0,T;H^2(\Om))$, and similarly for $\Gamma$, i.e.
	\bbk 
	\begin{equation}
	\label{eqs:space-reg}
		\begin{aligned}
			K_\Om(T;u) &= \|u\|_{H^2(0,T;H^2(\Om))} + \|\sosrcB\|_{L^2(0,T;H^2(\Om))} , \\
			K_\Ga(T;\ga u) &= \|\ga u\|_{H^2(0,T;H^2(\Ga))} +  \|\sosrcS\|_{L^2(0,T;H^2(\Ga))} , \\
		\end{aligned}
	\qquad\quad K(T;u,\ga u) = K_\Om(T;u) + K_\Ga(T;\ga u).
	\end{equation}
	\ebk 
	
	
	Note that by standard theory the
	estimate $\max_{0\leq t \leq T} \|u(t)\|_{X} \leq c \|u\|_{H^1(0,T;X)}$ holds, see, e.g.~\cite[Section~5.9.2]{Evans_PDE}.
	Therefore, we have
	\begin{align*}
	\sup_{0 \leq t \leq T} \norm{u(t)}_{H^2(\Om)} + \sup_{0 \leq t \leq T} \norm{\ga u(t)}_{H^2(\Ga)} \leq \bbk K(T;u,\ga u) , \ebk 
	\end{align*}
	\bbk where $K(T;u,\ga u)$ absorbed an embedding constant depending on $T$. \ebk 
	
\end{remark}
\subsection{Purely second-order dynamic boundary conditions}

The finite element approximation of \eqref{wave eqn - pure second-order} is given by \eqref{eq:sowt-h-var} and seeks the numerical solution $u_h$ in the space of piecewise linear finite elements
\begin{equation*}
V_h = S_h.
\end{equation*}
As described above, the semi-discrete problem is derived from the variational formulation by replacing the bilinear forms \eqref{eq:bilinear forms} with their discrete counterparts, i.e.~for $w_h,v_h\in V_h$,
\begin{equation}
\label{eq:bilinear forms - discrete}
\begin{aligned}
m_h(w_h,v_h) =&\ \int_{\Om_h} \! \! \! w_hv_h\, \d x + \mu  \int_{\Ga_h}\! \! \!  (\gamma_h w_h)(\gamma_h v_h)\, \d\sigma_h , \\
a_h(w_h,v_h) =&\ \int_{\Om_h} \! \! \! \nb w_h \cdot \nb v_h \, \d x
+\beta \int_{\Ga_h}\! \! \!  \nbgh w_h \cdot \nbgh v_h\, \d\sigma_h + \kappa \int_{\Ga_h}\! \! \!  (\gamma_h w_h)(\gamma_h v_h)\, \d\sigma_h
,
\end{aligned}
\end{equation}
where $\gamma_h$ denotes the trace operator onto $\Ga_h$, and $\nbgh$ is the discrete tangential gradient (defined in a piecewise sense by $\nbgh w_h = \nb \bar w_h -( \nb \bar w_h \cdot \nu_h) \nu_h$).
Finally, since $\sobfBsymb = 0$, we also have $\sobfBhsymb = 0$.

The lifted finite element approximation converges quadratically in the mesh size $h$ if measured in the a bulk--surface $L^2$ norm. 


\begin{theorem}[Purely second-order wave equations]
	\label{theorem: semi-discrete error bound: pure second}
	Let $u$ be  the solution of the wave equation with purely second-order dynamic boundary conditions \eqref{wave eqn - pure second-order}.
	
	If $\beta>0$ and the solution $u \in C^2(0,T;\HoB) \cap H^2(0,T;H^2(\Om))$ with $\ga u \in C^2(0,T;\HoS) \cap H^2(0,T;H^2(\Ga))$, then there is an $h_0 > 0$ such that for all $h\leq h_0$ the error between the solution $u$ and the linear finite element solution $u_h$ of \eqref{eq:sowt-h-var} (with \eqref{eq:bilinear forms - discrete}) satisfies the optimal second-order error estimate, for $0 \leq t \leq T$,
	\begin{equation}
	\|u_h^\ell\t - u\t\|_{L^2(\Om)} + \|\ga(u_h^\ell\t - u\t)\|_{L^2(\Ga)} \leq C(u,T) h^2.
	\end{equation}
	
	If $\beta=0$ and the solution $u \in C^2(0,T;\HoB) \cap H^2(0,T;H^2(\Om))$ and $\ga u \in C^2(0,T;\LtS)$, then we have the optimal second-order error estimate, for $0 \leq t \leq T$,
	\begin{equation}
	\label{eq:error estimate - pure second - beta = 0}
	\|u_h^\ell\t - u\t\|_{L^2(\Om)} + \|\ga(u_h^\ell\t - u\t)\|_{H^{-1/2}(\Ga)} \leq C(u,T) h^2 .
	\end{equation}
	In both cases, the constant $C(u,T)>0$ depends on $K(T;u,\ga u)$ from \eqref{eqs:space-reg} and grows linearly in the final time $T$, but it is independent of $h$ and $t$.
\end{theorem}

\subsection{Wave equations with advective dynamic boundary conditions}
\label{section:semi-discrete advective dynbc}

For the semi-discretisation of the wave equation with \emph{advective dynamic boundary conditions}, we again use the space $V_h=S_h$, together with the semi-discrete bilinear forms $m_h$ and $a_h$ from \eqref{eq:bilinear forms - discrete}. 

For the discrete counter part of $\sobfBsymb$ from \eqref{eq:bilinear forms - advective}, which accounts for the advective effects, we use the inverse lift of the vector fields and define, for $w_h, v_h \in V_h$,
\begin{equation}
\label{eq:bilinear forms - advective discrete}
\begin{aligned}
\sobfBh{w_h}{v_h} =&\ \int_{\Om_h} \big ( \wdampB w_h  +  \advB^{-\ell} \cdot \grad  w_h \big) v_h \d x_h 
+ \int_{\Ga_h} \big ( \wdampS  \ga w_h +  \advS^{-\ell} \cdot \nbgh  w_h \big) \ga_h v_h \d\sigma_h .
\end{aligned}
\end{equation}

\begin{remark}
	\label{remark:discrete vector fields - I}
	To obtain the discrete vector fields for the discrete bilinear form $b_h$ \eqref{eq:bilinear forms - advective discrete}, the inverse lifted vector fields might be difficult to compute in practice. Alternatively, approximative vector fields can be used. If the discrete vector fields $\advBh$ and $\advSh$ are sufficiently close (in terms of $h$) to $\advB$ and $\advS$, then the following convergence estimate remains valid.
	We will return to this later on (see Remark~\ref{remark:interpolated-vector-fields}), when the necessary tools are introduced.
\end{remark}


For the finite element approximation of the  wave equation with advective dynamic boundary conditions, we prove the following convergence results.

\begin{theorem}[Advective boundary conditions]
	\label{theorem: semi-discrete error bound: advective}
	Let the solution of the wave equation with advective dynamic boundary conditions  \eqref{eq:wave eqn - advective} have the regularity $u \in C^2(0,T;\HoB) \cap H^2(0,T;H^2(\Om))$ and $\ga u  \in C^2(0,T;\HoS) \cap H^2(0,T;H^2(\Ga))$, then there is an $h_0 > 0$ such that for all $h\leq h_0$ the error between the solution $u$ and the linear finite element semi-discretisation $u_h$ of \eqref{eq:sowt-h-var} (with \eqref{eq:bilinear forms - advective discrete}) satisfies the error estimate of order $3/2$, for $0 \leq t \leq T$,
	\begin{equation}
	\|u_h^\ell\t - u\t\|_{L^2(\Om)} + \|\ga(u_h^\ell\t - u\t)\|_{L^2(\Ga)} \leq C(u,T) h^{3/2},
	\end{equation}
	and if $\advB=0$ (but $\advS$ not necessarily) we have the optimal-order error estimates
	\begin{equation}
	\|u_h^\ell\t - u\t\|_{L^2(\Om)} + \|\ga(u_h^\ell\t - u\t)\|_{L^2(\Ga)} \leq C(u,T) h^2 .
	\end{equation}
	In both cases, the constant $C(u,T)>0$ depends on $K(T;u,\ga u)$ from \eqref{eqs:space-reg} and grows linearly in the final time $T$, but it is independent of $h$ and $t$.
\end{theorem}

\subsection{Wave equations with \sdamp\ dynamic boundary conditions}

For the semi-discretisation of the wave equation with \emph{\sdamp\ dynamic boundary conditions}, we again use the space $V_h=S_h$, together with the semi-discrete bilinear forms $m_h$ and $a_h$ from \eqref{eq:bilinear forms - discrete}.

The discrete bilinear form with strong damping corresponding to \eqref{eq:bilinear forms - strong damping} reads, for $w_h, v_h \in V_h$,
\begin{equation}
\label{eq:bilinear forms - strong damping discrete}
b_h(w_h,v_h) = d_\Om \int_{\Om_h} \nb w_h \cdot \nb v_h \ \d x_h + d_\Ga \int_{\Ga_h} \nbgh w_h \cdot \nbgh v_h \ \d \sigma_h .
\end{equation}

\begin{theorem}[Strong damping]
	\label{theorem: semi-discrete error bound: strong damping}
	Let the solution of the wave equation with \sdamp\ dynamic boundary conditions \eqref{eq:wave eqn - strong damping} have the regularity  $u \in C^2(0,T;\HoB) \cap H^2(0,T;H^2(\Om))$ and $\ga u  \in C^2(0,T;\HoS) \cap H^2(0,T;H^2(\Ga))$, then there is an $h_0 > 0$ such that for all $h\leq h_0$ the error between the solution $u$ and the linear finite element semi-discretisation $u_h$ of \eqref{eq:sowt-h-var} (with \eqref{eq:bilinear forms - strong damping discrete}) satisfies the first-order error estimate, for $0 \leq t \leq T$,
	\begin{equation}
	\|u_h^\ell\t - u\t\|_{L^2(\Om)} + \|\ga(u_h^\ell\t - u\t)\|_{L^2(\Ga)} \leq C(u) h ,
	\end{equation}
	and if $\beta = d_\Ga / d_\Om$, i.e.~the ratio of the diffusive and damping coefficients in the bilinear forms $a$ and $b$ coincide, then we have the optimal-order error estimate
	\begin{equation}
	\|u_h^\ell\t - u\t\|_{L^2(\Om)} + \|\ga(u_h^\ell\t - u\t)\|_{L^2(\Ga)} \leq C(u) h^2 .
	\end{equation} 
	In both cases, the constant $C(u,T)>0$ depends on $K(T;u,\ga u)$ from \eqref{eqs:space-reg} and grows linearly in the final time $T$, but it is independent of $h$ and $t$.
\end{theorem}

\begin{remark}
	The strongly damped wave equation with \sdamp\ dynamic boundary conditions is governed by an analytic semigroup, cf.~\cite{GraL14}.  
	Since we treat \eqref{eq:wave eqn - strong damping} as a hyperbolic problem our estimate is probably suboptimal. We expect that the error of the finite element solution converges with $\Landau(h^2)$ as shown in \cite{LarTW91} for the strongly damped wave equation with homogeneous Dirichlet boundary conditions.         %
\end{remark}

\subsection{Wave equations with acoustic boundary conditions}
\label{section:semi-discrete acoustic b.c.}

The solution $\vu = (u,\delta)$ of \eqref{eqs:ex-acoustic-pde} consists of two functions, one in the bulk and one on the surface.
Therefore we introduce the boundary 
element space on the surface $\Ga_h$:
\begin{equation*}
S_h^\Ga = \ga_h S_h = \big\{ \ga_h v_h \mid v_h \in S_h \big\} ,
\end{equation*}
to approximate the surface function $\delta\colon [0,T]\times \Ga \to \R$.
Hence we seek to approximate $\vu (t)$, which belongs to $V=H^1(\Om) \times H^1(\Ga)$, in the bulk--surface finite element space
\begin{equation}
\label{eq:acoustic FEM space}
V_h = S_h \times S_h^\Ga .
\end{equation}
The finite element approximation of the wave equation with acoustic boundary conditions \eqref{eqs:ex-acoustic-pde} now reads:
Find $\vec u_h=(u_h,\delta_h) \colon [0,T] \to V_h$ such that, for all $\vec v_h \in V_h$,
\begin{equation}
\label{eq:wave eqn - discrete acoustic}
\begin{aligned}
m_h(\ddot{\vec u}_h(t),\vec v_h) + b_h(\dot{\vec u}_h(t),\vec v_h) + a_h(\vec u_h(t),\vec v_h) &= m_h(\widetilde I_h \vec f(t),\vec v_h) , 
\\
\vec u_h(0) &= \intpo \vec u_0 ,  \qquad 
\dot{\vec u}_h(0) = \intpo \vec u_1 ,
\end{aligned}
\end{equation}
where the interpolation of any $\vv = (v,\psi)$ is understood componentwise \mbox{as $\intpo \vv = (\intpo^\Om v, \intpo^\Ga \psi) \in V_h$}, \bbk where $\intpo^\Om$ and $\intpo^\Ga$ are the standard interpolation operators on $\Om$ and $\Ga$, respectively. \ebk 

The discrete counterparts of the continuous bilinear forms \eqref{eqs:acoustic-bc-ingredient-var-form} are given by, for $\vec w_h = (w_h,\omega_h), \vec v_h = (v_h,\psi_h) \in \vec V_h$,
\begin{subequations}
	\label{eqs:acoustic-bc-ingredient-var-form discrete}
	\begin{align}
	\label{eq:acoustic discrete bilinear form - m}
	\ipHh{\vw_h}{\vec v_h}  &\defeq \int_{\Om_h} w_h v_h \ \d x_h + \int_{\Ga_h} \massS \omega_h \psi_h \ \d\sigma_h \\
	\sobfBh{\vw}{ \vec v_h } &\defeq  \diffB \int_{\Ga_h} (\ga_h w_h) \psi_h  - \omega_h (\ga_h v_h) \ \d\sigma_h , \\
	\label{eq:acoustic discrete bilinear form - a}
	\sobfh{\vw}{ \vec v_h } &\defeq \int_{\Om_h} \oscB w_h  v_h  + \diffB \nabla w_h \cdot \nabla v_h \ \d x_h 
	+ \int_{\Ga_h}  \stiffS \omega_h \psi_h + \diffS \nbgh \omega_h \cdot  \nbgh \psi_h \ \d\sigma_h , \\
	\ipHh{\intpo \vf}{\vec v_h} & \defeq \int_{\Om_h}  \intpo^\Om \sosrcB v_h \ \d x_h + \int_{\Ga_h} \intpo^\Ga \sosrcS \psi_h \ \d\sigma_h .
	\end{align}
\end{subequations}


As a consequence of the bulk--surface coupling our results show that the lifted finite element approximation only converges with $\Landau(h^{3/2})$ instead of  $\Landau(h^{2})$ as one would generally expect.

\begin{theorem}[Acoustic boundary conditions]
	\label{theorem: semi-discrete error bound: acoustic boundary conditions}
	Let the solution of the wave equation with acoustic boundary conditions \eqref{eqs:ex-acoustic-pde} have the regularity $u \in C^2(0,T;\HoB) \cap H^2(0,T;H^2(\Om))$ and $\delta \in C^2(0,T;\HoS) \cap H^2(0,T;H^2(\Ga))$, then there is an $h_0 > 0$ such that for all $h\leq h_0$ the error between the solution $u$ and the linear finite element semi-discretisation $\vec{u}_h=(u_h,\delta_h)$ of \eqref{eq:wave eqn - discrete acoustic} (with \eqref{eqs:acoustic-bc-ingredient-var-form discrete}) satisfies the error estimate of order $3/2$, for $0 \leq t \leq T$,
	\begin{equation}
	\label{eq:error estimate: acoustic bc}
	\|u_h^\ell\t - u\t\|_{L^2(\Om)} + \|\delta_h^\ell\t - \delta\t\|_{L^2(\Om)} \leq C(u,\delta,T) h^{3/2} ,
	\end{equation}
	where the constant $C(u,\delta,T)>0$ depends on $K(T;u,\delta)$ from \eqref{eqs:space-reg}, and grows linearly in the final time $T$, but it is independent of $h$ and $t$.
\end{theorem}

\section{Abstract error analysis}
\label{section:abstract error analysis}

This section is devoted to the abstract error analysis of \eqref{eq:sowt-h-var} with respect to \eqref{eq:sowt-var}. 
Our goal is to derive error bounds in terms of the $H$ norm which can later be used to show convergence rates for concrete examples.

The section is structured as follows:
First, we will show in Section~\ref{sec:stability energy norm} that the stability estimate derived by the \emph{usual} energy technique cannot lead to the optimal convergence rate of $\Landau(h^2)$. 
Then, in Section~\ref{section:stability est in weak norms}, we prove an alternative stability estimate in a weak norm by an adapted energy technique.
A discrete version of this estimate is shown in Section~\ref{section:sd stability}.
The main part of the error analysis is then presented in Section~\ref{section: error analysis and Ritz} where we derive an abstract error bound in terms of Ritz projection errors, geometric errors and data errors. In Section~\ref{section:FEM error analysis} we then prove estimates for these errors separately for each of the four cases.
\bbk Although the bilinear forms $a$ and $m$ are the same for all of the above problems (very similar for \eqref{eqs:ex-acoustic-pde}), the abstract analysis does not exploit this fact, it only uses the abstract assumptions from Section~\ref{section:abstract framework}. \ebk

\subsection{Stability estimate in the energy norm}
\label{sec:stability energy norm}

Let us consider the purely second-order version of \eqref{eq:sowt-var}:
\begin{subequations}   \label{eq:purely 2nd sowt-var}
	\begin{align}
	\dpV{\ddot u(t)}{v} + \sobf{u(t)}{v}  &= \dpV{\sosrc(t)}{v} \qquad \forall v \in V, \\
	u(0) &= \soiv, \quad \dot u(0) = \sois. 
	\end{align}
\end{subequations}
The following result yields a stability estimate in the $H$ norm for any sufficiently smooth solution.
\begin{proposition}
	Let $u \in C^2([0,T];H)$ be a solution of \eqref{eq:purely 2nd sowt-var}. Then
	\begin{align}
	\normH{u\t}^2 \leq C \bigl(    \normH{\dot u\t}^2 + \normVa{u\t}^2\bigr )\leq  C\ee \Big( \normH{\sois}^2 + \normVa{\soiv}^2 + T \int_0^t \normH{\sosrc(s)}^2  \ \d{s}  \Big) 
	\end{align}
	for $0 \leq t \leq T$.
\end{proposition}
\begin{proof}
	First note that due to $H \simeq H^*$ and $\ddot u \in H$, we have
	\begin{align*}
	\dpV{\ddot u}{v} = \ipH{\ddot u}{v} \quadfor v \in V.
	\end{align*}
	To derive the stability estimate, we test \eqref{eq:purely 2nd sowt-var} with $v = \dot u\t$.
	Together with
	\begin{equation*}
	\ipH{\dot u\t}{u\t} = \half \diff \normH{u\t}^2 \quadand 
	\sobf{\dot u\t}{u\t} = \half \diff \normVa{u\t}^2 , 
	\end{equation*}
	we thus obtain, with the Cauchy--Schwarz and Young's inequalities,
	\begin{align*}
	\half \diff \normH{\dot u\t}^2 + \half \diff \normVa{u\t}^2 
	\leq  &\ \normH{\sosrc\t} \normH{\dot u\t} 
	\leq  \frac{T}{2} \normH{\sosrc\t}^2 + \frac{1}{2T} \normH{\dot u\t}^2 	 .
	\end{align*}
	Integrating both sides and then applying Gronwall's inequality yields the classical stability bound
	\begin{align}
	\label{stability}
	\normH{\dot u\t}^2 + \normVa{u\t}^2 \leq &\ \ee \Big( \normH{\dot u(0)}^2 + \normVa{u(0)}^2 + T \int_0^t \normH{\sosrc(s)}^2 \ \d s  \Big) 	.
	\end{align}
	The first estimate of the claim is a consequence of $V \dembd H$ and $\normV{\cdot} \sim \normVa{\cdot}$.

\end{proof}

To obtain an error bound on the basis of \eqref{stability}, one would apply its discrete version to the error equation where the defect $d_h$ would play the role of $\sosrc$.
Thus $\normHh{d_h} \leq c h^2$ is necessary for second-order convergence.
However, as we will see later, the lower-order terms in the geometric estimates do not allow a second-order estimate in the discrete $L^2$ norm. 
Furthermore, this approach would require that the discrete initial value satisfies $\normVa{\uh^\ell(0) - u(0)} \leq c h^2$ which is only possible if we compute the Ritz map of $u_0$ and set it as a starting value for $u_h$.

To avoid this additional computational effort and obtain second-order convergence, we use another stability estimate, which measures the error of the initial data and for the defect in weaker norms.
A similar approach has been used in \cite{BakB79} to show $L^2$ convergence rates for a finite element discretisation of the wave equation with homogeneous classical boundary conditions.

\subsection{Stability estimate in weak norms}
\label{section:stability est in weak norms}

In addition to the canonical norm on $\Vdual$, we define
\begin{align*}
\normVadual{f} \defeq \sup_{v\neq 0} \frac{\dpV{f}{v}}{\normVa{v}} \quadfor f \in \Vdual.
\end{align*}

As a consequence of this definition and \eqref{eq:operators A and B def}, \bbk (which, together with coercivity, implies that the inverse operator $\soinv:\Vdual \to V$), we have, \ebk for $f\in \Vdual$,
\begin{align*}
\normVadual{f} 
&=  \sup_{\substack{v \in V \\ v\neq 0}} \frac{\dpV{f}{v}}{\normVa{v}}  =  \sup_{\substack{v \in V \\ v\neq 0}} \frac{\sobf{\soinv f}{v}}{\normVa{v}}  = \normVa{\soinv f},
\end{align*}
where the last equality follows from the fact the inner products are maximized by linear dependent elements.
Since then
\begin{align*}
\normVadual{f}^2 = \normVa{\soinv f}^2 
= \sobf{\soinv f}{\soinv f}  
=\dpV{\so \soinv f}{\soinv f} 
=\dpV{f}{\soinv f},
\end{align*}
the bilinear form $\dpV{\cdot}{\soinv \cdot}$ is the inner product in $\Vdual$ which induces $\normVadual{\cdot}$.

The weak stability estimate is a key step in proving the weak well-posedness result from Theorem~\ref{thm:w-p}.
In the following lemma, we derive the same stability estimate in a different way using energy techniques.

\begin{proposition} \label{prop:stability with velocity}
	Let $u$ be a solution of \eqref{eq:sowt-var} which satisfies \eqref{eq:weak regularity}.
	Then
	\begin{equation}
	\label{eq:stability-wave-velocity}
	\normVadual{\dot u \t + \soB  u\t}^2 + \normH{u\t }^2 
	\leq \ee^{\max\{1, 2T\cqmb\}} 
	\left (	\normVadual{\sois + \soB  \soiv}^2 + \normH{\soiv}^2 + T \int_0^t \normVadual{\sosrc(s)}^2 \d s \right ) ,
	\end{equation}
	for $0 \leq t \leq T$.
\end{proposition}

We remark here, that for all four examples from Section~\ref{section:wave eqn with dynbc} we have $\cqmb=0$.

\smallskip
\begin{proof}
	First observe that the left-hand side of \eqref{eq:sowt-var} can be rewritten as
	\begin{align*}
	\dpV{\ddot u\t}{v} + \sobfB{\dot u\t}{v} +  \sobf{u\t}{v}
	&=         \dpV{\left ( \dot u\t + \soB  u\t \right )^{\dotr}}{v} +\dpV{\so u\t}{v} .
	\end{align*}
	Testing \eqref{eq:sowt-var} with \bbk $v=\soinv (\dot u\t + \soB u\t) \in V$, where by~\eqref{eq:weak regularity} and $\soinv:\Vdual \to V$ it is seen that the test function is in $V$. \ebk 
	
	We therefore obtain 
	\begin{align*}
	\MoveEqLeft{\dpV{\sosrc\t}{\soinv ( \dot u\t + \soB u\t )}} & \\
	&  =   \dpV{\left ( \dot u\t + \soB  u\t \right )^{\dotr}}{\soinv \left ( \dot u\t + \soB u\t \right )} + \dpV{\so u\t}{ \soinv \left ( \dot u\t + \soB u\t \right )}  \\
	& =  \half \diff  \normVadual{\dot u\t + \soB  u\t}^2 + \half \diff \normH{u\t}^2  + \dpV{ \soB u\t}{u\t},
	\end{align*}
	where we used that $\dpV{\so u}{ \soinv \dot u} = \ipH{u}{\dot u}$ and that $\dpV{\cdot}{\soinv \cdot}$ is the inner product on $\Vdual$ which induces $\normVadual{\cdot}$.
	Since the bilinear form $\sobfBsymb\ccdot + \cqmb \ipHsymb\ccdot$ is monotone by \eqref{eq:quasi monotonicity of b}, we infer that
	\begin{align*}
	\MoveEqLeft{ \half \diff  \normVadual{\dot u\t + \soB  u\t}^2 + \half \diff \normH{u\t}^2 } & \\
	&  \leq \half \diff  \normVadual{\dot u\t + \soB  u\t}^2 + \half \diff \normH{u\t}^2 + \dpV{ \soB u\t}{u\t}   + \cqmb \normH{u\t}^2 \\
	& =  \dpV{\sosrc\t}{\soinv \left ( \dot u\t + \soB u\t \right )}  + \cqmb \normH{u\t}^2 \\
	& \leq   \tfrac{T}{2} \normVadual{\sosrc\t}^2 +  \tfrac{1}{2} \max \{ 1/T, 2 \cqmb \} \big ( \normVadual{ \dot u\t + \soB u\t }^2  + \normH{u\t}^2 \big ) ,
	\end{align*}
	where the last estimate is again shown by using Cauchy--Schwarz and Young's inequalities. 
	Finally, with Gronwall's inequality, we obtain the stated stability bound.
\end{proof}

We further note here that for wave equations with standard boundary conditions (but not in an abstract setting and without a velocity term) similar estimates have been shown, by choosing special test functions, started by the works of \cite{Dupont}, \cite{Baker} and \cite{BakB79}. In the case of strong damping (with homogeneous Dirichlet boundary conditions) similar techniques were used by \cite{LarTW91}.

\subsection{Semi-discrete stability estimate in discrete weak norms}
\label{section:sd stability}



This section is dedicated to the discrete weak norm stability estimate of the general semi-discrete problem \eqref{eq:sowt-h-var}. 
As in the proof of Proposition~\ref{prop:stability with velocity}, the discrete weak stability estimate is shown by testing the semi-discrete problem with 
\bbk \begin{equation}
\label{eq:test function semidiscrete}
	v_h = \sorh^{-1} ( \dot u_h + \soBrh \uh ) \in V_h,
\end{equation}
which is indeed in the finite dimensional space by \eqref{eq:A_h defintion}. \ebk 
We obtain the following semi-discrete stability result.

\begin{samepage}
	\begin{proposition}	\label{proposition: stability of error with velocity}
		Let $\uh$ be a solution of the semi-discrete problem \eqref{eq:sowt-h-var}. Then
		\begin{multline}
		\label{eq:semi-discrete stability of error with velocity}
		\normVhdual{\dot u_h\t + B_h u_h\t}^2  + \normHh{u_h\t}^2 \\
		\leq \ \ee^{\max\{1, 2T\cqmbh\}} \Big( \normVhdual{\dot u_h(0) + B_h u_h(0)}^2 + \normHh{u_h(0)}^2 + T \int_0^t \normVhdual{\sosrch(s)}^2 \d s \Big) ,
		\end{multline}
		for $0 \leq t \leq T$.
	\end{proposition}
	For the semi-discrete problems of our four examples, described in Section~\ref{section:main results}, we have $\cqmbh=0$.
\end{samepage}
\subsection{Error analysis}
\label{section: error analysis and Ritz}


To obtain an upper bound for the error $u  - \lift u_h$, we split it into the error of the Ritz map and $e_h = u_h - \Ritz u$
\begin{equation*}
\label{eq:error decomposition}
\begin{aligned}
\normH{u \t - \lift u_h \t} 
&\leq   \normH{ u \t - \Ritzl u \t } + \normH*{\lift{(\Ritz u\t -  u_h \t)}} \\
&\leq   \normH{ u \t - \Ritzl u \t } + C \normHh*{e_h \t},
\end{aligned}
\end{equation*}
where we used the norm equivalence \eqref{eq:norm equiv} for the second inequality.
In the rest of this section, we will show that $\normHh*{e_h \t}$ is bounded from above by a combination of Ritz map, geometric and data errors, while the first term already is a  Ritz map error.

For that purpose, we insert $\Ritz u$ into the semi-discrete problem \eqref{eq:sowt-h-var} and define the defect (or semi-discrete residual) $\defect \colon [0,T] \to V_h$, for all $\vh \in \sods$, by
\begin{equation}
\label{eq:defect definition}
\ipHh{\Ritz \ddot  u }{\vh} + \sobfBh{ \Ritz \dot  u }{ \vh} + \sobfh{ \Ritz u}{\vh} 
= \ipHh{\sosrch}{\vh} + \ipHh{\defect}{\vh} . 
\end{equation}
Subtracting this from the semi-discrete weak problem \eqref{eq:sowt-h-var} yields the error equation
\begin{equation*} 
\ipHh{ \ddot  e_h }{\vh} + \sobfBh{ \dot  e_h }{ \vh} + \sobfh{ e_h}{\vh} = - \ipHh{\defect}{\vh} \qquad \forall \vh \in \sods.
\end{equation*}
To obtain an upper bound for $\normHh*{e_h \t}$, we apply the stability estimate from Proposition~\ref{proposition: stability of error with velocity} to the error equation. Using Young's inequality for products then gives
\begin{align}
\normHh{e_h\t } 
& \leq \normVhdual{\dot e_h \t + \soBrh  e_h\t} + \normHh{e_h\t} \notag \\ 
& \leq C \ee^{T\cqmbh} \left (    \normVhdual{\dot e_h (0) + \soBrh  e_h(0)}^2 + \normHh{e_h (0) }^2 + T \int_0^t \normVhdual{\defect(s)}^2 \d s \right)^{1/2} \label{eq:estimate eh}.
\end{align}
Since the errors at the initial time $t=0$ are bounded by
\begin{align*}
\normHh{e_h (0) } 
&{} = \normHh{\Ritz \soiv - \soivh }  \leq C \normH{\Ritzl \soiv - \soiv}  + C \normH{\intpol \soiv - \soiv } \\
\intertext{and}
\normVhdual{\dot e_h (0) + \soBrh  e_h(0)}  
&{}= \normVhdual{(\Ritz \sois - \soish) + \soBrh  (\Ritz \soiv - \soivh)}
\\
&{}\leq C \normHh{\Ritz \sois - \soish}  + \normVhdual{ \soBrh  (\Ritz \soiv - \soivh)}
\\
&{}\leq C \big (\normHh{\Ritzl \sois - \sois} +  \normHh{\intpol \sois - \sois} \big )
+  \normVhdual{ \soBrh  (\Ritz \soiv - \soiv^{-\ell})} +  \normVhdual{ \soBrh  (\soiv^{-\ell} - \soivh)},
\end{align*}
where the last terms are further bounded from above since for $w_h = \Ritz \soiv - \soiv^{-\ell}$ and $w_h =\soiv^{-\ell} - \soivh$
\begin{align*}
\normVhdual{ \soBrh  w_h} 
\leq &\ \sup_{0 \neq v_h \in V_h} \frac{\Delta b(w_h,v_h)}{\|v_h\|_h} + c \sup_{0 \neq v \in V} \frac{b(w_h^\ell,v)}{\|v\|} 
\leq \ \|\Delta b(w_h,\cdot)\|_{\star,h} + c \|B w_h^\ell\|_{\star} .
\end{align*}

\begin{samepage}
	We have altogether shown
	\begin{subequations}
		\label{eq:abstract estimate}
		\begin{align}
		\label{eq:abstract estimate energy est}
		&\ \normH{u \t - \lift u_h \t} \leq \normH{ u \t - \Ritzl u \t } 
		+ C \ee^{\cqmbh T} \bigg ( \eps_0^2 + T \int_0^t \normVhdual{\defect(s)}^2 \d s \bigg )^{1/2} ,
		\intertext{where}
		\label{eq:initial data error}
		&\ \begin{aligned}
		\eps_0 = &\ \normH{\intpol \sois - \sois} + \normH{\intpol \soiv - \soiv } + \normH{\Ritzl \sois - \sois} + \normH{\Ritzl \soiv - \soiv} \\
		&\ + \|\Delta b(\Ritz \soiv - \soiv^{-\ell},\cdot)\|_{\star,h} + \|\Delta b(\soivh - \soiv^{-\ell},\cdot)\|_{\star,h} \\
		&\ + c \|B(\Ritzl \soiv - \soiv)\|_{\star} + c \|B(I_h \soiv - \soiv)\|_{\star} .
		\end{aligned}
		\end{align}
	\end{subequations}
\end{samepage}

To obtain convergence rates from this abstract estimate, it remains to study the defect $\defect$ further.
%
\begin{lemma}  \label{lemma:abstract defect bound}
	The defect $\defect$ defined by \eqref{eq:defect definition} satisfies the following estimate
	\begin{equation}
	\label{eq:defect bound}
	\begin{aligned}
	\normVhdual{\defect} 
	\leq  C \Big( & \bbk \normVadual{\sosrc - \intpol \sosrc} \ebk \\
	&\ + \bbk \normVhdual{\DeltaipH{\Ritz \ddot u}{\cdot}} + \normVhdual{\DeltasobfB{\Ritz \dot u}{\cdot}}  + \normVhdual{\DeltaipH{\sosrch}{\cdot}} \ebk \\
	&\ + \bbk \normVadual{\Ritzl \ddot u - \ddot u} + \normVadual{\soB (\Ritzl \dot u - \dot u)} \ebk \Big ) 
	\end{aligned}
	\end{equation}
	for $0 \leq t \leq T$.
\end{lemma}

\begin{proof}
	We subtract \eqref{eq:sowt-var} with $v = v_h^\ell$ from \eqref{eq:defect definition} to compute the defect
	\begin{equation}
	\label{eq:defect equation with pairs}
	\begin{aligned}
	\ipHh{\defect}{\vh} = &\ \ipHh{\Ritz \ddot u}{v_h} - \ipH{\ddot u}{\lift v_h} \\
	&\ + \sobfBh{\Ritz \dot u}{v_h} - \sobfB{\dot u}{\lift v_h} \\
	&\ + \sobfh{\Ritz u}{v_h} - \sobf{u}{\lift v_h} \\
	&\ + \ipH{\sosrc}{\lift v_h} - \ipHh{\sosrch}{v_h} .
	\end{aligned}
	\end{equation}
	These pairs are then estimated separately.
	For the first pair, we have
	\begin{align*}
	\ipHh{\Ritz \ddot u}{v_h} - \ipH{\ddot u}{\lift v_h}
	& =  \ipHh{\Ritz \ddot u}{v_h} - \ipH{\Ritzl \ddot u}{\lift v_h} +  \ipH{\Ritz \ddot u - \ddot u}{\lift v_h} \\
	& \leq  - \DeltaipH{\Ritz \ddot u}{v_h} +  \normVadual{\Ritz \ddot u - \ddot u} \normVa{\lift v_h} ,
	\end{align*}
	where we used that $\ipH{w}{v}= \dpV{w}{v} \leq \norm{w}_{V^*} \normV{v} \leq C \normVadual{w} \normVa{v}$.       
	For the second pair, we have
	\begin{align*}
	\sobfBh{\Ritz \dot u}{v_h} - \sobfB{\dot u}{\lift v_h}  
	&= \sobfBh{\Ritz \dot u}{v_h} - \sobfB{\Ritzl \dot u}{\lift v_h} + \sobfB{\Ritzl \dot u - \dot u}{\lift v_h}  \\
	&=- \DeltasobfB{\Ritz \dot u}{v_h} + \dpV{\soB(\Ritzl \dot u - \dot u)}{\lift v_h}  \\
	&\leq - \DeltasobfB{\Ritz \dot u}{v_h} + C \normVadual{\soB(\Ritzl \dot u - \dot u)} \normVa{\lift v_h}  .
	\end{align*}
	The third pair vanishes by definition of the Ritz map, cf.~\eqref{Ritz}. 
	
	\noindent For the fourth pair, we have
	\begin{align*}
	\ipH{\sosrc}{\lift{v_h}}- \ipHh{\sosrch}{v_h}
	& = \ipH{\sosrc - \intpol \sosrc}{\lift v_h}  + \ipH{ \intpol f}{\lift v_h} - \ipHh{ \sosrch}{v_h} \\
	& \leq \normVadual{\sosrc - \intpol \sosrc} \normVa{\lift v_h}    + \DeltaipH{ \sosrch}{v_h}
	\end{align*}
	Since by \eqref{eq:norm equiv}
	\begin{align*}
	\normVhdual{\defect }  
	= \sup_{v_h \in V_h} \frac{\ipHh{\defect}{v_h}}{\normVh{v_h}}
	\leq C  \sup_{v_h \in V_h} \frac{\ipHh{\defect}{v_h}}{\normVa{\lift v_h}},
	\end{align*}
	the claim follows upon combining the above estimates.
\end{proof}

\section{Finite element error analysis}
\label{section:FEM error analysis}

In this section, we prove the convergence results stated in Section~\ref{section:main results}.
The main part of all these proofs was already done in Section~\ref{section:abstract error analysis} where we derived the abstract a priori estimate \eqref{eq:abstract estimate} and showed an upper bound for the defect in Lemma~\ref{lemma:abstract defect bound}.
To obtain convergence rates, it remains to estimate the error components in terms of the mesh width $h$.
This can be done by using approximation results from the literature and using the properties of the first-order term $\soB$ which, as it will turn out, lead to the different convergence rates appearing in Theorem~\ref{theorem: semi-discrete error bound: pure second}--\ref{theorem: semi-discrete error bound: acoustic boundary conditions}.

The results of the previous section can be also summarised as: By substituting the estimates \eqref{eq:initial data error} and \eqref{eq:defect bound} into \eqref{eq:abstract estimate energy est}, the $L^2$ error of the semi-discrete solution is bounded by
\begin{align*}
	\normH{u\t - u_h^\ell\t} \leq &\ \textnormal{interpolation errors} \\
	&\ + \textnormal{geometric approximation errors} \\
	&\ + \textnormal{Ritz map errors} . 
\end{align*}
In the next section we show that these errors are indeed small.

\bbk 
\subsection{Interpolation errors and a boundary layer estimate}

Before we turn to the proof of Theorem~\ref{theorem: semi-discrete error bound: pure second}, we collect error estimates of the nodal interpolations in the bulk and on the surface, and a technical result. 
From Section~\ref{section:wave eqn with dynbc}, and \ref{section:finite element method} we recall our assumptions on the bulk and the surface, and on their discrete counterparts: the bounded domain $\Om \subset\R^d$ ($d=2$ or $3$) has an (at least) $C^2$ boundary $\Gamma$; the quasi-uniform triangulation $\Om_h$ (approximating $\Om$) whose boundary $\Ga_h := \pa \Om_h$ is an interpolation of $\Ga$.

\begin{lemma}
	\label{lemma:nonpoly-interpolation}
	For $v\in H^2(\Om)$, such that $\gamma v \in H^2(\Ga)$, we denote by $I_h v\in V_h^\ell$ the lift of the nodal finite element interpolation $\widetilde{I}_h v\in V_h$. Then the following estimates hold:
	\begin{enumerate}
		\item[(i)] Interpolation error in the bulk; see \cite{Bernardi,ElliottRanner}:
		\begin{equation*}
		\|v - I_h v\|_{L^2(\Om)} + h \|\nb(v - I_h v)\|_{L^2(\Om)} \leq C h^2 \|v\|_{H^2(\Om)}.
		\end{equation*}
		\item[(ii)] Interpolation error on the surface; see \cite{Dziuk88}:
		\begin{equation*}
		\|\gamma(v - I_h v)\|_{L^2(\Ga)} + h \|\nb_\Ga(v - I_h v)\|_{L^2(\Ga)}
		\leq Ch^2 \|\gamma v\|_{H^2(\Ga)}.
		\end{equation*}
	\end{enumerate}
\end{lemma}

The following technical result helps to estimate norms on a layer of triangles around the boundary.
\begin{lemma}[{\cite[Lemma~6.3]{ElliottRanner}}]
	\label{lemma:layer est}
	For all $v\in H^1(\Om)$ the following estimate holds:
	\begin{equation}
	\|v\|_{L^2(B_h^\ell)} \leq C h^\Half \|v\|_{H^1(\Om)},
	\end{equation}
	where $B_h^\ell$ denotes layer of lifted elements which have a boundary face.
\end{lemma}
\ebk 

\subsection{Purely second-order wave equation}
\subsubsection{Geometric errors}

The bilinear forms $a$ and $a_h$, from \eqref{eq:bilinear forms} and \eqref{eq:bilinear forms - discrete}, satisfy the following geometric approximation estimate.
\begin{lemma}[{\cite[Lemma~3.9]{dynbc}}]
	\label{lemma:geometric errors}
	\bbk For the bilinear forms \eqref{eq:bilinear forms} and their discrete counterparts \eqref{eq:bilinear forms - discrete} we have the estimates, \ebk for any $v_h,w_h \in S_h$, 
	\begin{align*}
		& |a(v_h^\ell,w_h^\ell) - a_h(v_h,w_h)| \leq \ Ch \|\nb v_h^\ell\|_{L^2(B_h^\ell)} \, \|\nb w_h^\ell\|_{L^2(B_h^\ell)}
		\\ &\qquad + C h^2 \bigg(\|\nb v_h^\ell\|_{L^2(\Om)} \, \|\nb w_h^\ell\|_{L^2(\Om)} + \beta \|\nb_\Ga v_h^\ell\|_{L^2(\Ga)} \, \|\nb_\Ga w_h^\ell\|_{L^2(\Ga)} + \kappa \|\gamma v_h^\ell\|_{L^2(\Ga)} \, \|\gamma w_h^\ell\|_{L^2(\Ga)}\bigg),
		\\
		& |m(v_h^\ell,w_h^\ell) - m_h(v_h,w_h)| \leq \ Ch \|v_h^\ell\|_{L^2(B_h^\ell)} \, \|w_h^\ell\|_{L^2(B_h^\ell)} \\
		&\qquad\qquad\qquad\qquad\qquad\quad + C h^2 \Big(\|v_h^\ell\|_{L^2(\Om)} \, \|w_h^\ell\|_{L^2(\Om)} +\mu \|\gamma v_h^\ell\|_{L^2(\Ga)} \, \|\gamma w_h^\ell\|_{L^2(\Ga)}\Big) .
	\end{align*}
	
\end{lemma}

\subsubsection{Error estimates for the Ritz map}
From \cite[Lemma~3.13 and 3.15]{dynbc} we recall the following estimates for the error of the Ritz map.
Note the weaker norm on $\Ga$ for $\beta=0$ due to a lack of boundary regularity of solutions of the Poisson equation with Neumann boundary conditions.

\newcommand{\vr}{v-R_hv}
\begin{lemma}
	\label{lemma: Ritz error est} 
	The error of the Ritz map \eqref{Ritz} corresponding to the bilinear form $a$ from \eqref{eq:bilinear forms}  satisfies the following second-order bounds:
	
	For $\beta =0$:
	$$
	\|\vr\|_{L^2(\Om)} + \|\gamma(\vr)\|_{H^{-1/2}(\Ga)} \leq C h^2 \, \|v\|_{H^2(\Om)} ,
	$$
	where the constant $C$ is independent of $h$  and $v\in H^2(\Omega)$.
	
	For $\beta>0$:
	$$
	\|\vr\|_{L^2(\Om)} +  \|\gamma(\vr)\|_{L^2(\Ga)} \leq C h^2 \, \big( \|v\|_{H^2(\Om)}  + \|\gamma v\|_{H^2(\Gamma)} \big),
	$$
	where the constant $C$ is independent of $h$  and $v\in H^2(\Omega)$ with $\gamma v \in H^2(\Gamma)$, but depends on $\beta>0$.
\end{lemma}

\subsubsection{Proof of Theorem~\ref{theorem: semi-discrete error bound: pure second}}
For the proof, we simply apply the abstract results from the Section~\ref{section:abstract error analysis} and use the estimates for the error of the Ritz map, the interpolation error and the geometric errors from above.

We start by considering the error of the Ritz map in the dual norm.
For $\beta = 0$ and $v \in H^2(\Om)$, it follows from \cite[Section~3.5.2]{dynbc} that
\begin{subequations}
	\begin{align} \label{eq:Ritz beta = 0}
	\normVadual{v - \Ritzl v} \leq \big( \norm{v - \Ritzl v}_{L^2(\Om)}^2 + \norm{v - \Ritzl v}_{H^{-1/2}(\Ga)}^2 \big)^\half  \leq c h^2
	\end{align}
	and, for $\beta > 0$ and for $v \in H^2(\Om)$ with $\gamma v \in H^2(\Ga)$, 
	\begin{align} \label{eq:Ritz beta > 0}
	\normVadual{v - \Ritzl v} \leq \normH{v - \Ritzl v} = \big( \norm{v - \Ritzl v}_{L^2(\Om)}^2 + \norm{v - \Ritzl v}_{L^2(\Ga)}^2 \big)^\half \leq c h^2, 
	\end{align}
\end{subequations}
where we used Lemma~\ref{lemma: Ritz error est} in the last inequality for both estimates.

Now we can further estimate the upper bound for the defect from Lemma~\ref{lemma:abstract defect bound}:
The Ritz map error for $\ddot u \in H^2(\Om)$ is bounded due to our previous arguments, the first-order terms do not appear since $\sobfBsymb = 0$ and hence $\soB = 0$, and the $L^2$ norm error estimate of the interpolation Lemma~\ref{lemma:nonpoly-interpolation} yields
\begin{equation*}
\normVadual{ \sosrc - I_h \sosrc } \leq \normH{ \sosrc - I_h \sosrc } \leq c h^2 .
\end{equation*}
The geometric errors can be bounded as follows: Combining Lemma~\ref{lemma:geometric errors} and \ref{lemma:layer est} yields $\DeltaipH{w_h}{v_h} \leq  c h^2  \norm{w_h^\ell}_{H^1(\Om)}\norm{v_h^\ell}_{H^1(\Om)}$.
Therefore we obtain with
\eqref{eq:norm equiv} that
\begin{equation}
\label{eq:delta m estimate}
\normVhdual{\DeltaipH{w_h}{\cdot}} = \sup_{v_h \in V_h} \frac{\DeltaipH{w_h}{v_h}}{\normVh{v_h}} 
\leq  c h^2  \normVh{w_h}.
\end{equation}
Since, first, the norm equivalence \eqref{eq:norm equiv} and the interpolation estimate from Lemma~\ref{lemma:nonpoly-interpolation} yield for  $f \in H^2(\Om) \cap V $
\begin{align*}
\normVh{\sosrch} \leq   \normVa{\sosrc - I_h \sosrc} + \normVa{\sosrc} \leq ch + c ,
\end{align*}
and, second, $\Ritz \in \mathcal L (V,V_h)$ and $\ddot u \in V$, the geometric error is bounded by
\begin{align*}
\normVhdual{\DeltaipH{\Ritz \ddot u}{\cdot}} +	\normVhdual{\DeltaipH{\sosrch}{\cdot}}  \leq c h^2 .
\end{align*}
Altogether, we showed that under the given assumptions $\normVhdual{\defect} \leq c h^2$ for $\beta \geq 0$.

For the errors in the initial data $\eps_0$ only the first line of \eqref{eq:initial data error} is present, hence Ritz map and interpolation error estimates yields $\eps_0 \leq c h^2$.

Finally, for $\beta > 0$, we apply \eqref{eq:Ritz beta > 0} for the error in the Ritz map and  $\normVhdual{\defect} \leq c h^2$  to the right-hand side of \eqref{eq:abstract estimate energy est} and obtain the stated, optimal-order convergence bound
\begin{equation*}
\normH{u \t - \lift u_h \t} \leq C h^2 .		
\end{equation*}
If $\beta = 0$, then observe that for $e = u - \lift u_h $ and $e_h = \Ritz u  -  u_h$
\begin{align*}
\norm{e}_{L^2(\Om)} + \norm{e}_{H^{-1/2}(\Ga)}
\leq  \norm{u - \Ritzl u}_{L^2(\Om)} + \norm{u - \Ritzl u}_{H^{-1/2}(\Ga)} + C \normHh{e_h}.
\end{align*}
The errors of the Ritz map are bounded by \eqref{eq:Ritz beta = 0} and $\normHh{e_h}$ satisfies \eqref{eq:estimate eh}.
Therefore, we obtain from the estimate for the defect the optimal-order convergence bound
\begin{equation*}
\norm{u \t - \lift u_h \t}_{L^2(\Om)} +  \norm{u \t - \lift u_h \t}_{H^{-1/2}(\Ga)} \leq C h^2 .		
\end{equation*}
\qed 

\subsection{Advective \BC}
\subsubsection{Geometric errors}

The bilinear form containing the advective terms \eqref{eq:bilinear forms - advective} and \eqref{eq:bilinear forms - advective discrete} satisfy the following geometric approximation estimate, shown in \cite[Lemma~7.3]{Hip17}.
\begin{lemma}
	\label{lemma:geometric-error-advective}
	For sufficiently small $h \leq h_0$, and for any $v_h,w_h \in S_h$ we have the estimate
	\begin{align*}
	|b(w_h^\ell,v_h^\ell) - b_h(w_h,v_h)|  
	\leq &\ c h \bigg(
	\norm{\advB}_{L^\infty(B_h^\ell)} \|\nb w_h^\ell\|_{L^2(B_h^\ell)} \, \|v_h^\ell\|_{L^2(B_h^\ell)} 
	+  \wdampB \|w_h^\ell\|_{L^2(B_h^\ell)}\|v_h^\ell\|_{L^2(B_h^\ell)} \bigg) \\ 
	&\ + c h^2 \bigg(  \|\nbg w_h^\ell\|_{L^2(\Ga)}\|\ga v_h^\ell\|_{L^2(\Ga)} +   \|\ga w_h^\ell\|_{L^2(\Ga)} \, \|\ga v_h^\ell\|_{L^2(\Ga)} \bigg) ,
	\end{align*}
	where the constant $c$ is independent of $h$, but depends on the $L^\infty$ norms of the coefficient functions $\wdampB, \wdampS, \advB$ and $\advS$.
\end{lemma}

We now give more details on the approximative vector fields discussed in Remark~\ref{remark:discrete vector fields - I}.
\begin{remark}
	\label{remark:interpolated-vector-fields}
	One way to avoid computing the inverse lift of the \bbk continuous \ebk vector fields is to use their interpolations $\advBh = \widetilde{I}_h \advB$ and $\advSh = \widetilde{I}_h \advS$ instead. 
	
	Then the above geometric approximation estimate of Lemma~\ref{lemma:geometric-error-advective} holds, using the interpolation error estimate, see the proof of \cite[Lemma~7.3]{Hip17}. 
	The quasi-monotonicity of $b_h$ is shown using the assumptions which guarantee that the original bilinear form $b$ is monotone (i.e.~the conditions in \eqref{eq:advective eq coefficient assumptions}), by proving $0 \leq \cqmbh \leq c h$ below, meaning that the stability estimate for the semi-discrete equation \eqref{eq:semi-discrete stability of error with velocity} holds with an exponent which is almost zero.
	
	First note that \bdh for a differentiable vector field $F \colon \Omega \to \R^2$ and a differentiable coordinate transformation $G \colon \Omega' \to \Omega$ from some other domain $\Om' \subset \R^2$ to $\Omega$, \edh it follows by the chain rule that
	\begin{equation}
	\label{eq:divergence calculation}
		\begin{aligned}
			\div (F \circ G) = \sum_{i=1,2} e_i^T D(F \circ G) e_i 
			& = \sum_{i=1,2} e_i^T (DF \circ G) DG e_i 
			\\
			&= (\div F) \circ G  + \sum_{i=1,2} e_i^T \big ( (DF \circ G) (DG - I)   \big )e_i ,
		\end{aligned}
	\end{equation}
	where $e_i \in \R^d$ denotes unit vector along the $i$th coordinate axis.
	Now, \bbk by setting $\Om'=\Om_h$ and using \eqref{eq:divergence calculation}, \ebk let $G_h \colon \Omega_h \to \Omega$ the smooth homeomorphism such that the lift of a function $v_h \colon \Omega_h \to \R$ is given by $v_h^\ell = v \circ G_h$, \bbk cf.~\eqref{eq:bulk mapping}. \ebk 
	Then we obtain that the divergence of the inverse lift of the vector field $\advB$ is given by
	\begin{align}\label{eq:div-of-unlift}
	\div \advB^{-\ell} = \div (\advB \circ G_h) 
	= (\div \advB)^{-\ell} + \sum_{i=1,2} e_i^T \big ( (D \advB)^{-\ell} (DG_h - I)   \big )e_i  .
	\end{align}
	
	Recall that the bilinear form $b$ is monotone due to the conditions \eqref{eq:advective eq coefficient assumptions}.
	However, the bulk condition
	\begin{equation}
	\label{eq:bulk condition recalled}
	0 \leq \min_{x\in \Omega} \Big( \wdampB - \frac 12 \div \advB(x) \Big) 
	= \wdampB - \frac 12 \max_{x\in \Omega} \div \advB(x) ,
	\end{equation}
	is not necessarily satisfied for the interpolated vector fields (or analogously for the surface condition).
	Instead we have that $\sobfBhsymb + \cqmbh \ipHhsymb$ is monotone for
	\begin{equation*}
	\cqmbh = - \min_{x\in \Om_h} \Big( \wdampB - \frac 12 \div( \intpo \advB) (x) \Big) = - \wdampB + \frac 12 \max_{x\in \Om_h} \div( \intpo \advB) (x) . 
	\end{equation*}
	If $\cqmbh$ is negative then the semi-discrete bilinear form  $\sobfBhsymb$ is monotone, and $\sobfBhsymb + \cqmbh \ipHhsymb\geq 0$ holds with $\cqmbh=0$.
	Hence we can assume $\cqmbh \geq 0$ and by \eqref{eq:bulk condition recalled} we have
	\begin{align*}
	\cqmbh \leq &\ \cqmbh + \wdampB - \frac 12 \max_{x\in \Omega} \div \advB(x) \\
	\leq &\ - \wdampB + \frac 12 \max_{x\in \Om_h} \div( \intpo \advB) (x) + \wdampB - \frac 12 \max_{x\in \Omega} \div \advB(x) \\
	=  &\  \frac 12 \max_{x\in \Om_h} \div( \intpo \advB) (x) - \frac 12 \max_{x\in \Om_h} (\div \advB)^{-\ell}(x)  \\
	\leq &\ \frac 12  \max_{x\in \Om_h}\left|  \div( \intpo \advB) (x)-   (\div \advB)^{-\ell}(x) \right| 
	\\
	\leq & \ \norm{ \div( \intpo \advB) - (\div \advB)^{-\ell}}_{L^\infty(\Om_h)} .
	\end{align*}
	Using \eqref{eq:div-of-unlift} and the interpolation error estimate yields
	\begin{align*}
	\norm{  \div( \intpo \advB) - (\div \advB)^{-\ell}}_{L^\infty(\Om_h)} 
	&\leq  \norm{\div(\intpo \advB -  \advB^{-\ell})}_{L^\infty(\Om_h)} + C \norm{(D \advB)^{-\ell}}_{L^\infty(\Om_h)} \norm{DG_h - I}_{L^\infty(\Om_h)}
	\\
	&\leq  \norm{\intpo \advB -  \advB^{-\ell}}_{W^{1,\infty}(\Om_h)} + C \norm{D \advB}_{L^\infty(\Om)} \norm{DG_h - I}_{L^\infty(\Om_h)}
	\\
	& \leq c h \norm{\advB}_{W^{2,\infty}(\Om_h)} .
	\end{align*}
	Altogether we proved that $0 \leq \cqmbh \leq c h$ for interpolated bulk vector fields.
	
	Analogously, the surface condition in \eqref{eq:advective eq coefficient assumptions} might also fail for the interpolated vector fields. Repeating the argument above for this case, we have that
	\begin{equation*}
	\cqmbh = \max \left \{ - \min_{x\in \Om_h} \Big( \wdampB - \frac 12 \div( \intpo \advB) (x) \Big), - \min_{x\in \Ga_h} \Big( \wdampS - \frac 12 \big ((\nu \cdot \intpo \advB) (x) - \divS (\intpo \advS ) (x) \big )\right \}. 
	\end{equation*}
	also satisfies the bounds $0 \leq \cqmbh \leq c h .$
	
	Therefore, with interpolated vector fields, Proposition~\ref{proposition: stability of error with velocity} and hence Theorem~\ref{theorem: semi-discrete error bound: advective} holds with a constants which grow like $e^{\cqmbh T}=e^{c h T}$ in the final time $T$.
\end{remark}

\subsubsection{Proof of Theorem~\ref{theorem: semi-discrete error bound: advective}}

We proceed analogously to the proof of Theorem~\ref{theorem: semi-discrete error bound: pure second}.

First we prove the case with advection in the bulk and on the boundary (a), and then prove the result without bulk advection, i.e.~with $\advB=0$, (b).

Note that the error estimate \eqref{eq:Ritz beta > 0} for the Ritz map still applies in both situations.

(a) In order to show $\normVhdual{\defect} \leq c h^{3/2}$, it is only left to consider the first-order terms from Lemma~\ref{lemma:abstract defect bound} containing $\soB$. 
The other terms were already treated in the proof of Theorem~\ref{theorem: semi-discrete error bound: pure second}.

To estimate  $\normVadual{\soB (\Ritzl \dot u - \dot u)}$ for $\sobfBsymb$ defined in \eqref{eq:bilinear forms - advective}, we use 
\begin{align*}
\sobfB{w}{v} \leq c \normH{w} \normVa{v}, \qquad w \in H, \: v \in V,
\end{align*}
which follows from integration by parts and the assumptions \eqref{eq:advective eq coefficient assumptions} on the coefficient functions.
Therefore, we have by definition of the dual norm and \eqref{eq:Ritz beta > 0},
\begin{equation}
\label{eq: Ritz map B dual error - advective}
\normVadual{\soB (\Ritzl \dot u - \dot u)} \leq  c  \normH{\Ritzl \dot u - \dot u} \leq c h^2 .
\end{equation}

For an upper bound for $\DeltasobfBsymb$, we use the geometric estimates stated in Lemma~\ref{lemma:geometric-error-advective}.
Applying Lemma~\ref{lemma:layer est} to further estimate the boundary layer norms for $\lift v_h$ then yields
\begin{equation}
\label{eq:delta b estimate - 3/2 order}
\begin{aligned}
\abs{\DeltasobfB{w_h}{v_h}} 
& \leq c h^{3/2} \norm{\lift w_h}_{H^1(\Om)} \norm{\lift v_h}_{H^1(\Om)} 
+ c h^2 \norm{\lift w_h}_{H^1(\Ga)} \norm{\lift v_h}_{H^1(\Ga)} 
\\
&\leq c h^{3/2} \normVa{\lift w_h} \normVa{\lift v_h} .
\end{aligned}
\end{equation}
Therefore, the geometric error for $\sobfBsymb$ converges with
\begin{align*}
\normVadual{\DeltasobfB{\Ritz \dot u}{\cdot}} \leq c h^{3/2}.
\end{align*}

Altogether, we obtain for a sufficiently small $h \leq h_0$
\begin{equation*}
\normVhdual{\defect}  \leq C h^{3/2},
\end{equation*}
with a constant $C$ independent of $h$ but depending on Sobolev norms of the solution $u$ (and also its time derivatives).

For the errors in the initial data \eqref{eq:initial data error} we now have $\eps_0 \leq c h^{3/2}$, by similar arguments as above: using the bound \eqref{eq: Ritz map B dual error - advective} and the geometric estimate \eqref{eq:delta b estimate - 3/2 order}, and Ritz map and interpolation error estimates.

We again recall that the error $u - u_h^\ell$ was estimated in terms of the defect and errors in the initial data \eqref{eq:abstract estimate}.
The combination of this estimate with the above results yields the convergence bound:
\begin{equation*}
\normH{u \t - \lift u_h \t} \leq C h^{3/2}	.
\end{equation*}

(b) If there is no advection in the bulk, i.e.~$\advB=0$, clearly the first term in the right-hand side of the estimate in Lemma~\ref{lemma:geometric-error-advective} vanishes, hence, using Lemma~\ref{lemma:layer est}, we have $\Landau (h^2)$ estimate in \eqref{eq:delta b estimate - 3/2 order}. Then, by the same techniques as before we then obtain the defect estimate $\normVhdual{\defect}  \leq C h^2$ and initial data error $\eps_0 \leq c h^2$, and hence the optimal-order convergence bound:
\begin{equation*}
\normH{u \t - \lift u_h \t} \leq C h^2 .
\end{equation*}
\qed 

\subsection{\Sdamp\ \dynBC}

\subsubsection{Geometric errors}

Since the bilinear forms $b$ and $b_h$ defined in \eqref{eq:bilinear forms - strong damping} and \eqref{eq:bilinear forms - strong damping discrete} contain the same terms as $a$ and $a_h$,
the following geometric approximation estimate is an immediate consequence of Lemma~\ref{lemma:geometric errors}.

\begin{lemma}
	\label{lemma:geometric-error-strong-damping}
	For sufficiently small $h \leq h_0$, and for any $v_h,w_h \in S_h$ we have the estimates
	\begin{align*}
	|b(w_h^\ell,v_h^\ell) - b_h(w_h,v_h)| \leq &\ c h \|\nb w_h^\ell\|_{L^2(B_h^\ell)} \|\nb v_h^\ell\|_{L^2(B_h^\ell)} + c h^2 \|\nb_\Ga w_h^\ell\|_{L^2(\Ga)} \|\nb_\Ga v_h^\ell\|_{L^2(\Ga)} ,
	\end{align*}
	where the constant $c$ is independent of $h$, but depends on $d_\Om$ and $d_\Ga$.
\end{lemma}

\subsubsection{Proof of Theorem~\ref{theorem: semi-discrete error bound: strong damping}}

We proceed analogously to the previous proofs.

We prove the general case and the case with coefficients satisfying $\beta = d_\Ga / d_\Om$ separately.

(a) Again, note that the error estimate \eqref{eq:Ritz beta > 0} for the Ritz map still applies in this situation and that it is only left to consider the first-order terms from Lemma~\ref{lemma:abstract defect bound} to prove a defect estimate.

Using $\soB \in \mathcal L(V,\Vdual)$ and by the $\|\cdot\|$ norm error estimate for the Ritz map \cite[Lemma~3.1]{dynbc} we find that for $\dot u \in H^2(\Om) \cap V$
\begin{equation*}
\normVadual{\soB (\Ritzl \dot u - \dot u)} \leq c \normVa{\Ritzl \dot u - \dot u} \leq c h .
\end{equation*}

In addition, the geometric error of $\sobfBsymb$ is bounded as $\Landau (h)$ such that altogether, for a sufficiently small $h \leq h_0$, we have
\begin{equation*}
\normVhdual{\defect} \leq C h ,
\end{equation*}
with a constant $C$ independent of $h$ but depending on Sobolev norms of the solution $u$ (and also its time derivatives).

For the errors in the initial data \eqref{eq:initial data error} we have $\eps_0 \leq c h$, by similar arguments as for the defect above, and using Ritz map and interpolation error estimates.

Again recalling the error estimate \eqref{eq:abstract estimate}, and combining it with the above inequalities we obtain the stated, convergence bound
\begin{equation*}
\normH{u \t - \lift u_h \t} \leq C h .		
\end{equation*}

(b) In case the ratio of the diffusion and damping coefficients coincide
$$\beta = d_\Ga / d_\Om,$$
the bilinear forms $a$ and $b$ (and the semi-discrete counterparts) also coincide up to a constant and some lower order terms (which are related to $m$), using the definitions \eqref{eq:bilinear forms} and \eqref{eq:bilinear forms - strong damping}, we have
\begin{align*}
b(w,v) = &\ d_\Om \bigg( \int_\Om \nb w \cdot \nb v \ \d x + \frac{d_\Ga}{d_\Om} \int_\Ga \nbg w \cdot \nbg v \ \d \sigma\bigg) \\
= &\ d_\Om \Big( a(w,v) \Big) - d_\Om \kappa \int_\Ga (\ga w) (\ga v) \ \d \sigma .
\end{align*}
In the proof of Lemma~\ref{lemma:abstract defect bound}, in particular in \eqref{eq:defect equation with pairs}, not only the pair for $a$ vanishes due to the definition of the Ritz map \eqref{Ritz}, but the pair for $b$ as well up to a mass term on the boundary, using the identity from above. Therefore, the critical term from part (a) does not appear at all, but instead we have to bound the boundary mass pair, similarly as we have done for \eqref{eq:delta m estimate}, and obtain
\begin{equation*}
\Big| \int_{\Ga_h} (\ga_h \Ritz \dot u) (\ga_h v_h) \ \d \sigma_h - \int_\Ga (\ga \dot u) (\ga v_h^\ell) \ \d \sigma \Big| \leq c h^2 .
\end{equation*}
The rest of the proof is finished as part (a) and yields a defect estimate $\normVhdual{\defect} \leq C h^2$ and initial data error bound $\eps_0 \leq c h^2$, and hence an optimal-order error estimate: 
\begin{equation*}
\normH{u \t - \lift u_h \t} \leq C h^2 .		
\end{equation*}

\qed 

\subsection{Acoustic \BC}

\subsubsection{Geometric and interpolation errors}
In this section, we treat the geometric errors in the bilinear forms form the equation with acoustic boundary conditions. 
Although, the following estimates are a straightforward generalisation of the results from \cite{DziukElliott_L2} and \cite{ElliottRanner}, we present the proofs to avoid any confusion due to the vector valued functions.

\begin{lemma}	\label{lemma:geometric-error-acoustic}
	For sufficiently small $h \leq h_0$, and for any $\vec w_h = (w_h,\omega_h), \vec v_h =(v_h,\psi_h) \in \vec S_h$ we have the estimates
	\begin{align*}
	|m(\vec w_h^\ell,\vec v_h^\ell) - m_h(\vec w_h,\vec v_h)| 
	\leq &\ c h \|w_h^\ell\|_{L^2(B_h^\ell)} \, \|v_h^\ell\|_{L^2(B_h^\ell)} + c h^2 \|\omega_h^\ell\|_{L^2(\Ga)} \, \|\psi_h^\ell\|_{L^2(\Ga)} , 
	\\
	|a(\vec w_h^\ell,\vec v_h^\ell) - a_h(\vec w_h,\vec v_h)| 
	\leq &\ 
	c h \|w_h^\ell\|_{H^1(B_h^\ell)} \, \|v_h^\ell\|_{H^1(B_h^\ell)}  + c h^2 \|\omega_h^\ell\|_{H^1(\Ga)} \, \|\psi_h^\ell\|_{H^1(\Ga)} , 
	\\
	|b(\vec w_h^\ell,\vec v_h^\ell) - b_h(\vec w_h,\vec v_h)| 
	\leq &\ c h^2 \|\ga w_h^\ell\|_{L^2(\Ga)} \, \|\psi_h^\ell\|_{L^2(\Ga)} 
	+ c h^2 \|\omega_h^\ell\|_{L^2(\Ga)} \, \|\ga v_h^\ell\|_{L^2(\Ga)} , \\
	\leq &\ c h^2 \|\vec w_h^\ell\|_{H^1(\Om)\times L^2(\Ga)} \, \|\vec v_h^\ell\|_{H^1(\Om)\times L^2(\Ga)} ,
	\end{align*}
	with constants independent of $h$, but depending on  $\diffS$, $\diffB$, $\massS$, $\oscB$ and $\stiffS$.
\end{lemma}

\begin{proof}
	The first and the second estimate can be shown in the same way as Lemma~\ref{lemma:geometric errors}.
	For the last inequality, using \cite[Lemma~5.5, (5.13)]{DziukElliott_L2} and using that the lift, see Section~\ref{section:finite element method}, satisfies $(\ga_h v_h)^\ell = \ga (v_h^\ell)$ for all $v_h \in V_h$, we therefore obtain
	\begin{align*}
	|b(\vec w_h^\ell,\vec v_h^\ell) - b_h(\vec w_h,\vec v_h)| 
	\leq &\ \diffB \Big| \int_\Ga \ga w_h^\ell  \psi_h^\ell  - \omega_h^\ell  (\ga v_h^\ell) \d\sigma - \int_{\Ga_h} (\ga_h w_h) \psi_h  - \omega_h \ga_h v_h \d\sigma \Big| 
	\\
	\leq &\ \diffB \Big| \int_\Ga (\ga w_h)^\ell  \psi_h^\ell \d\sigma - \int_{\Ga_h} (\ga_h w_h) \psi_h \d\sigma \Big| 
	+ \diffB \Big| \int_{\Ga_h} \omega_h (\ga_h v_h) \d\sigma - \int_\Ga \omega_h^\ell  (\ga \vh)^\ell \d\sigma   \Big| 
	\\
	\leq &\ c h^2 \|\ga w_h^\ell\|_{L^2(\Ga)} \, \|\psi_h^\ell\|_{L^2(\Ga)} 
	+ c h^2 \|\omega_h^\ell\|_{L^2(\Ga)} \, \|\ga v_h^\ell\|_{L^2(\Ga)} 
	\\
	\leq &\ c h^2 \|\vec w_h^\ell\|_{H^1(\Om)\times L^2(\Ga)} \, \|\vec v_h^\ell\|_{H^1(\Om)\times L^2(\Ga)}.
	\end{align*}
\end{proof}

For the numerical discretisation of the wave equation with acoustic boundary conditions, we need two interpolation operators.
In addition to the bulk interpolation $\intpo^\Om \colon H^2(\Om) \to S_h$ which was defined in Lemma~\ref{lemma:nonpoly-interpolation}, we introduce the nodal interpolation of surface functions
\begin{align*}
\intpo^\Ga \colon H^2(\Ga) \to S_h^\Ga .
\end{align*}
As the mesh of $\Ga_h$ is given by the boundary nodes of $\calT_h$, the following identity-via-traces holds 
for $v \in H^2(\Om)$ with $\gamma (v) \in H^2(\Ga)$
\begin{align*}
\intpo^\Ga ( \gamma (v)) = \gamma_h ( \intpo^\Om v).
\end{align*}
Using the abbreviations
\begin{align*}
H^k = H^k(\Om) \times H^k(\Ga)\quad\text{for} \quad k\geq 1, \, \quadand L^2 = L^2(\Om) \times L^2(\Ga)
\end{align*}
we define the bulk--surface interpolation operator $\intpo \colon H^2 \to V_h$  componentwise as $\intpo (v,\psi) = (\intpo^\Om v, \intpo^\Ga \psi)$.
Similarly as the estimates of Lemma~\ref{lemma:nonpoly-interpolation}, or by \cite[Proposition~5.4]{ElliottRanner}, the lifted interpolation $\intpol \vv = \lift{(\intpo \vv)}$ has the following approximation property.

\begin{lemma} \label{lemma:nonpoly-interpolation-acoustic}
	Let $\vv\in H^2$.
	Then the interpolation error of $\intpol \vv$ is bounded by
	\begin{align*}
	\norm{\vv - \intpol \vv}_{L^2} +  h \norm{\vv - \intpol \vv}_{H^1}
	\leq C h^2 \norm{\vv}_{H^2} .
	\end{align*}
\end{lemma}

\subsubsection{Error estimates for the Ritz map}

We recall the definition of the Ritz map from \eqref{Ritz} in the notation for acoustic boundary conditions: 
For $\vec w = (w,\omega) \in V$ we define $\Ritz \vec w \in V_h$ by
\begin{equation}
\label{eq:Ritz for acoustic}
\sobfh{\Ritz \vec w}{ \vec v_h } = \sobf{\vec w}{\lift{\vec{v}_h}} , \qquad \textnormal{ for all } \quad \vec v_h = (v_h,\psi_h) \in V_h 
\end{equation}  
and set $\Ritzl \vec w =\lift{(\Ritz \vec w)}$.

It is crucial to note that $\sobfsymb$ and $\sobfhsymb$ defined in \eqref{eq:acoustic bilinear form - a} and \eqref{eq:acoustic discrete bilinear form - a} do not couple bulk variables with surface variables. Therefore the Ritz map is given by a component-wise application of Ritz maps in the bulk and on the surface, i.e.~we have 
\begin{align*}
\Ritz \vec w =(\Ritz^\Om w,\Ritz^\Ga \omega)  .
\end{align*}
Accordingly, we define the lift of components as $\Ritzl^\Om  w =\lift{(\Ritz^\Om w)}$ and $\Ritzl^\Ga \omega =\lift{(\Ritz^\Ga \omega)}$, and hence $\Ritzl \vec w =(\Ritzl^\Om w,\Ritzl^\Ga \omega)$.

The second-order error estimate for the Ritz map thus follows from a combination of existing results, cf.~\cite[Section~5.4]{BreS08} for the bulk, and \cite{LubichMansour_wave} for the surface Ritz map.

\newcommand{\er}{\vec w - \Ritzl \vec w}
\begin{lemma}
	\label{lemma:Ritz acoustic}
	The error of the Ritz map \eqref{eq:Ritz for acoustic}, with the bilinear forms \eqref{eq:acoustic bilinear form - a} and \eqref{eq:acoustic discrete bilinear form - a}, on a smooth domain satisfies the following bounds, for $h \leq h_0$ with $h_0$ sufficiently small,
	\begin{align}
	\|\er\| \leq &\ C h \|\vw\|_{H^2}, \\
	|\er| \leq &\ C h^2 \|\vw\|_{H^2}, 
	\end{align}
	where the constants $C>0$ are independent of $h$ and $\vw \in H^2$.        
	%
\end{lemma}

\subsubsection{Proof of Theorem~\ref{theorem: semi-discrete error bound: acoustic boundary conditions}}
The structure of the proof is the same as in the previous sections. However, due to $\sobfBsymb$ containing the bulk--surface coupling, we present the complete analysis.

We start by estimating the defect. The Ritz map error estimate for acoustic boundary conditions Lemma~\ref{lemma:Ritz acoustic}, and interpolation estimates of Lemma~\ref{lemma:nonpoly-interpolation-acoustic} yield for $\ddot{\vec{u}}, \vec \sosrc \in H^2$ 
\begin{align*}
\normVadual{\Ritzl \ddot{\vec{u}} - \ddot{\vec{u}}} &\leq  C  \normH{\Ritzl \ddot{\vec{u}} - \ddot{\vec{u}}} \leq c h^2 ,
\\
\normVadual{\intpol \vec \sosrc - \vec \sosrc}  & \leq C \normH{\intpol \vec \sosrc - \vec \sosrc} \leq c h^2 .
\end{align*}

The term $\normVadual{\soB (\Ritzl \dot u - \dot u)}$ is estimated directly, using \eqref{eq:acoustic bilinear form - b}, Lemma~\ref{lemma:Ritz acoustic}, and the following version of the trace inequality, for a function $w \in H^1(\Om)$ and for an arbitrary but fixed $\eps > 0$ we have the trace inequality
\begin{equation}
\label{eq:trace inequality with eps}
\|\ga w\|_{L^2(\Ga)} \leq \eps \|\nb w\|_{L^2(\Om)} + c \frac{1}{\eps} \|w\|_{L^2(\Om)}.
\end{equation}
The proof of this trace inequality uses an $\eps$-Young's inequality instead of the standard one, cf.~\cite[(1) in Section~5.5]{Evans_PDE}, but otherwise it is the same as usual.

By choosing $\eps=h^{1/2}>0$ in \eqref{eq:trace inequality with eps}, for $\vv =(v,\psi) \in V = H^1$, we obtain
\begin{align*}
b(R_h \dot{\vec{u}} - \dot{\vec{u}} , \vec v) 
\leq  &\ c \Big| \int_\Ga \big( \ga (R_h^\Om \dot{u} - \dot{u}) \big) \psi  - \big( R_h^\Ga \dot{\delta} - \dot{\delta} \big) (\ga v) \ \d\sigma \Big| \\
\leq &\ c \| \ga (R_h^\Om \dot{u} - \dot{u})\|_{L^2(\Ga)}  \|\psi\|_{L^2(\Ga)} + c \|R_h^\Ga \dot{\delta} - \dot{\delta}\|_{L^2(\Ga)} \|\ga v\|_{L^2(\Ga)} \\
\leq &\ c \Big( h^{1/2} \|\nb (R_h^\Om \dot{u} - \dot{u})\|_{L^2(\Om)} + \frac{c}{h^{1/2}} \|R_h^\Om \dot{u} - \dot{u}\|_{L^2(\Om)} \Big)  \|\psi\|_{L^2(\Ga)} + c \|R_h^\Ga \dot{\delta} - \dot{\delta}\|_{L^2(\Ga)} c \|v\|_{H^{1/2}(\Om)} \\
\leq &\ c \Big( h^{1+1/2} \|\dot{u}\|_{H^2(\Om)} + c h^{2-1/2} \|\dot{u}\|_{H^2(\Om)} \Big)  \|\psi\|_{L^2(\Ga)} + c \|R_h^\Ga \dot{\delta} - \dot{\delta}\|_{L^2(\Ga)} c \|v\|_{H^{1/2}(\Om)} \\
\leq &\ \big( c h^{3/2} \|\dot u\|_{H^2(\Om)} + c h^2 \|\dot{\delta}\|_{H^2(\Ga)} \big) \big(\|v\|_{H^1(\Om)} + \|\psi\|_{H^1(\Ga)} \big) \\
\leq &\ c h^{3/2} \|\dot{\vec{u}}\|_{H^2} \|\vec v\| .
\end{align*}

By the definition of the dual norm we have
\begin{equation}
\label{eq:Ritz dual B estimate}
\normVadual{\soB (\Ritzl \dot{\vec{u}} - \dot{\vec{u}})} \leq c h^{3/2} .
\end{equation}

For the convergence of the geometric errors note that by Lemma~\ref{lemma:geometric-error-acoustic} we have the bound
\begin{equation}
\label{eq:Delta b estimate for acoustic bc}
\DeltasobfB{\vwh}{\vvh} \leq c h^2 \normVa{\lift{\vwh}} \normVa{\lift{\vwh}}.
\end{equation}
Therefore, it follows as in the proof of Theorem~\ref{theorem: semi-discrete error bound: pure second} that for $\ddot{\vec{u}},\dot{\vu} \in V$ and $\vec \sosrc \in H^2 \cap V$
\begin{align*}
\normVhdual{\DeltaipH{\Ritz \ddot{\vec{u}}}{\cdot}}
+ \normVhdual{\DeltasobfB{\Ritz \dot{\vec{u}}}{\cdot}}
+ \normVhdual{\DeltaipH{\intpo \vec \sosrc}{\cdot}} \leq c h^2 .
\end{align*}

Altogether, we obtain the defect estimate, for a sufficiently small $h \leq h_0$,
\begin{equation*}
\normVhdual{\defect} \leq C h^{3/2},
\end{equation*}
with a constant $C$ independent of $h$, but depending on Sobolev norms of the solution $u$ (and also its time derivatives).

For the errors in the initial data \eqref{eq:initial data error} we again have $\eps_0 \leq c h^{3/2}$, by similar arguments used above to prove \eqref{eq:Ritz dual B estimate} and \eqref{eq:Delta b estimate for acoustic bc}, together with Ritz map and interpolation error estimates.

We again recall that the error $u - u_h^\ell$ was estimated in terms of the defect and initial data error \eqref{eq:abstract estimate}, the combination of this estimate with the above results yields the convergence bound:
\begin{equation*}
\normH{u \t - \lift u_h \t} \leq C h^{3/2} .		
\end{equation*}
\qed 

\section{Time discretisations}
\label{section:time discreteisations}

\subsection{Matrix--vector formulation}
We collect the nodal values of $u_h(\cdot,t) = \sum_{j=1}^N u_j(t) \phi_j(\cdot) \in V_h$ the solution of the semi-discrete problem \eqref{eq:sowt-h-var} into the vector $\bfu \t = (u_1\t,\ldots,u_N \t) \in \R^N$, and we define the matrices corresponding to the bilinear forms $m_h$, $a_h$ and $b_h$, respectively, and the load vector:
\begin{equation}
\label{eq:FEM matrices}
\begin{aligned}
\bfM|_{kj} = &\ m_h(\phi_j,\phi_k) , \\
\bfA|_{kj} = &\ a_h(\phi_j,\phi_k) , \\
\bfB|_{kj} = &\ b_h(\phi_j,\phi_k) , \\
\overline\bfb|_{k} = &\ m_h(\sosrch(\cdot,t),\phi_k) , \\
\end{aligned}
\qquad j,k = 1,\dotsc,N ,
\end{equation}
where $\phi_j$ ($j= 1,\dotsc,N$) are the basis functions of $V_h$. In the case of acoustic boundary conditions Section~\ref{section:semi-discrete acoustic b.c.} all functions in $V_h$ are vector valued, see \eqref{eq:acoustic FEM space}. \bbk In particular, the basis of $V_h = S_h \times S_h^\Ga$, from \eqref{eq:acoustic FEM space}, is the product of the bases of $S_h$ and $S_h^\Ga$. \ebk All matrices inherit their properties from their corresponding bilinear form, therefore, both matrices $\bfM$ and $\bfA$ are symmetric and positive definite, while the matrix $\bfB + \cqmbh \bfM$ (with $\cqmbh \geq 0$) is positive semi-definite, but can be non-symmetric. \bbk For problems with acoustic boundary conditions the above matrices are block diagonal, with the blocks containing the respective matrices of the bulk or the surface. \ebk 

Then the semi-discrete problem \eqref{eq:sowt-h-var} is equivalent to the following matrix--vector formulation:
\begin{equation}
\label{eq:matrix-vector form - general}
\begin{aligned}
\bfM \ddot\bfu\t + \bfB \dot\bfu\t + \bfA \bfu\t =&\ \overline\bfb\t ,
\end{aligned}
\end{equation}
with initial values $\bfu(0) = \bfu^0$ and $\dot\bfu(0) = \bfv^0$, where $\bfu^0$ and $\bfv^0$ collects the nodal values of the interpolations of $u_0$ and $u_1$.

The above second order system of ordinary differential equations can be written as the first order system, by introducing the new variable 
\bbk \begin{equation}
\label{eq:bfv the new variable}
\bfv\t = \bfM \dot\bfu\t + \bfB \bfu\t,
\end{equation}
collecting the nodal values of $v_h(\cdot,t) = \sum_{j=1}^N v_j(t) \phi_j(\cdot)$, we obtain \ebk 
\begin{equation}
\label{eq:matrix-vector form - first order}
\begin{aligned}
\dot\bfv\t = &\ - \bfA \bfu\t + \overline\bfb\t , \\
\dot\bfu\t = &\  \bfM\inv \bfv\t - \bfM\inv \bfB \bfu\t .
\end{aligned}
\end{equation}

Using the variable $\bfy\t = (\bfv\t,\bfu\t)^T \in \R^{2N}$, the block matrices
\begin{equation*}
\bfJ = \left(\begin{matrix}
0 & \Id_N \\ 
-\Id_N & 0
\end{matrix}\right) , \qquad
\bfH = \left(\begin{matrix}
\bfM\inv & 0 \\ 
0 & \bfA
\end{matrix}\right)  \andquad  
\bfHH = \left(\begin{matrix}
0 & -\bfM\inv \bfB \\ 
0 & 0
\end{matrix}\right) ,
\end{equation*}
and the load vector $\bfb\t =(\overline\bfb\t,0)^T$, and where $\Id_N$ denotes the identity matrix of $N=\text{dim}(V_h)$. The ODE system \eqref{eq:matrix-vector form - first order} is equivalent to the first order ODE system
\begin{equation}
\label{eq:first order form}
\dot\bfy\t = \bfJ\inv \big(\bfH + \bfHH \big) \bfy\t + \bfb\t ,
\end{equation}
with initial value $\bfy(0) = (\bfM\bfv^0 + \bfB \bfu^0,\bfu^0)^T$. The system is written in this form, since for $\bfB=0$ it is Hamiltonian. This (further) geometric structure will be used to show stability of the full discretisation.

\newcommand\bfS{{\mathbf S}}

We further introduce the matrix
\begin{equation}
\label{eq:def bfS}
\bfS = \left(\begin{matrix}
\bfA\inv & 0 \\ 
0 & \bfM
\end{matrix}\right) ,
\end{equation}
and the corresponding induced norm, for arbitrary $\bfy=(\bfv,\bfu)^T$,
\begin{equation}
\label{eq:S norm def}
\|\bfy\|_\bfS^2 = \bfy^T \bfS \bfy = \|\bfv\|_{\bfA\inv}^2 + \|\bfu\|_{\bfM}^2 = \|v_h\|_{\star,h}^2 + |u_h|_h^2 ,
\end{equation}
\bbk which, by comparing \eqref{eq:test function semidiscrete} and \eqref{eq:bfv the new variable}, fits perfectly to the norm of the weak norm energy estimate of Proposition~\ref{proposition: stability of error with velocity}. \ebk 

Along the proof of the stability bounds we need the following properties. The operator $\bfJ\inv \bfH$ is skew-symmetric with respect to the $\bfS$ inner product,  direct computation shows:
\begin{equation}
\label{eq:S product estimate}
\begin{aligned}
\bfy^T \bfS \bfJ\inv \bfH \bfy 
= &\ \left(\begin{matrix} \bfA\inv \bfv \\ \bfM \bfu \end{matrix}\right)^T \left(\begin{matrix} 0 & -\Id_N \\  \Id_N & 0 \end{matrix}\right) \left(\begin{matrix} \bfM\inv \bfv \\ \bfA \bfu \end{matrix}\right) \\
= &\ - \bfv^T \bfA\inv \bfA \bfu + \bfu^T \bfM \bfM\inv \bfv \\
= &\ 0 .
\end{aligned}
\end{equation}
While for $\bfHH$, using the quasi monotonicity of the bilinear form $b_h$, we have the inequality
\begin{equation}
\label{eq:S product estimate with hat H}
\begin{aligned}
\bfy^T \bfS \bfJ\inv \bfHH \bfy = &\ - \bfu^T \bfM \bfM\inv \bfB \bfu = - \bfu^T \big( \bfB + \cqmbh \Id_N \big)\bfu + \cqmbh \bfu^T \bfu \leq 0 + c \|\bfu\|_{\bfM}^2 \leq c \|\bfy\|_\bfS^2 .
\end{aligned}
\end{equation}

\bbk We note here that compared above cited papers (cf.,~in particular, \cite[Section~2.6 and 2.7]{Mansour_GRK}) the matrix $\bfS$ is chosen differently, but \eqref{eq:S product estimate} and \eqref{eq:S product estimate with hat H} hold similarly. \ebk


\subsection{Implicit Runge--Kutta methods}

The first-order system of ordinary differential equations \eqref{eq:first order form} is discretised in time using an $s$-stage implicit Runge--Kutta method. For a fixed time step size $\tau>0$, the method determines the approximations $\bfy^n$ and the internal stages $\bfY^{ni}$, for $n\tau \leq T$, by
\begin{samepage}
	\begin{subequations}
		\label{RK method}
		\begin{align}
		\bfY^{ni}  = &\ \bfy^n + \tau \sum_{j=1}^n a_{ij} \dot\bfY^{nj} ,  & i=1,\dotsc,s , \\
		\bfy^{n+1}  = &\ \bfy^n + \tau \sum_{j=1}^n b_j \dot\bfY^{nj} \\
		\intertext{where the internal stages satisfy}
		\dot\bfY^{nj} =&\ \bfJ\inv (\bfH + \bfHH) \bfY^{nj} + \bfb^{nj} ,    & j=1,\dotsc,s ,
		\end{align}
	\end{subequations}
\end{samepage}
with $\bfb^{nj} = \bfb(t_n+c_j\tau)$, ,and where $\dot\bfY^{nj}$ is not a time derivative, only a suggestive notation. The method is determined by its coefficient matrix $\AC=(a_{ij})_{i,j=1}^s$, weights $b=(b_i)_{i=1}^s$ and nodes $c=(c_i)_{i=1}^s$. 

In the following, we assume the Runge--Kutta method \eqref{RK method} to be \emph{algebraically stable}, i.e.~the coefficients $b_j \geq 0$ and the matrix with entries
\begin{align*}
b_i a_{ij} + b_j a_{ji} - b_ib_j \quad \textnormal{is positive semi-definite} .
\end{align*}
We also assume that the coefficient matrix is invertible $\AC\inv=(w_{ij})_{i,j=1}^s$. Furthermore, the Runge--Kutta method is \emph{coercive}, that is, there exists a positive definite diagonal matrix $\calD\in\R^{s \times s}$ and $\alpha>0$ such that
\begin{equation}
\label{eq:R-K coercivity}
\bfw^T \calD \AC\inv \bfw \geq \alpha \bfw^T \calD \bfw \quadfora \bfw \in \R^s .
\end{equation}
The diagonal matrix is explicitly given by $\calD = \textnormal{diag}(b)(\textnormal{diag}(c)\inv - \Id_s)$, see \cite{HairerWannerII}.

\smallskip
The Runge--Kutta methods based on Gauss (and also those on Radau IA and Radau IIA) collocation nodes are known to be algebraically stable and coercive, and to have a non-singular coefficient matrix. For more details on these concepts and such methods we refer to \cite[Chapter~IV]{HairerWannerII} and the references therein.

\subsection{Convergence of the full discretisation with Gauss--Runge--Kutta methods}

\bbk 
The proof of convergence clearly separates the issues of stability and consistency. The consistency analysis follows in the usual way by estimating defects, on the other hand stability, while as the semi-discrete case it also relies on energy estimates, is more involved and is carried out in detail.
\ebk 


\subsubsection{Stability}

We first show stability for implicit Runge--Kutta methods applied to the ODE system \eqref{eq:first order form}. The proof of the stability bound is a straightforward simplification of the corresponding results of \cite[Lemma~4.1]{Mansour_GRK}, or \cite[Section~3]{HP_RK}, \cite[Lemma~5.1--5.2]{HPS}, \cite[Lemma~4.3]{rkkato}, where more general problems than \eqref{eq:first order form} are considered.

\begin{lemma}[Stability]
	\label{lemma:RK stability}
	The error $\bfe^n = \bfy^n - \bfy(t_n)$ between the numerical solution obtained by an $s$-stage implicit Runge--Kutta method and $\bfy(t_n)$ the exact solution of \eqref{eq:first order form} satisfies the bound, for $n\tau \leq T$,
	\begin{equation*}
	\|\bfe^n\|_\bfS \leq C \Big(\|\bfe^0\|_\bfS^2 + \sum_{k=0}^{n-1} \sum_{j=1}^s \|\bfD^{kj}\|_\bfS^2 + \sum_{k=1}^n \|\bfd^{k}\|_\bfS^2 \Big)^{1/2} ,
	\end{equation*}
	where the defects $\bfD^{kj}$ and $\bfd^{k}$ are obtained by substituting the nodal values of the exact solution into the method \eqref{RK method}. The constant $C>0$ is independent of $h, \tau$ and $n$. 
\end{lemma}

\begin{proof}
	The proof is based on energy techniques for algebraically stable Runge--Kutta methods, using Lady Windermere's fan \cite[II.3 and I.7]{HairerWannerII}, and it is a straightforward simplification of the proof of Lemma~4.1 in \cite{Mansour_GRK}. A clear difference is the use of a different norm, induced by $\bfS$ here, which is due to the required weaker norms in the final estimates, \bbk compare the norms in \eqref{eq:semi-discrete stability of error with velocity} and \eqref{eq:S norm def}. \ebk Furthermore, the estimate (2.19) in \cite{Mansour_GRK}, analogous to \eqref{eq:S product estimate} here, is slightly different, but it is used in the same way. The additional $\bfHH$ is treated analogously as $\bfH$ but using the bound \eqref{eq:S product estimate with hat H}.
	
	\bigskip
	The values $\widetilde\bfy^n = \bfy(t_n)$ and $\widetilde{\bfY}^{nj}=\bfy(t_n+c_j\tau)$ of the exact solution of the ODE \eqref{eq:first order form} only satisfies the \eqref{RK method} up to some defects:
	\begin{equation}
	\label{eq:RK defect def}
		\begin{aligned}
			\widetilde\bfY^{ni}  = &\ \widetilde\bfy^n + \tau \sum_{j=1}^n a_{ij} \dot{\widetilde{\bfY}}^{nj} + \bfD^{nj}, \qquad i=1,\dotsc,s, \\
			\widetilde\bfy^{n+1}  = &\ \widetilde\bfy^n + \tau \sum_{j=1}^n b_j \dot{\widetilde{\bfY}}^{nj} + \bfd^{n+1} \\
			\dot{\widetilde{\bfY}}^{nj} = &\ \bfJ\inv (\bfH + \bfHH) \widetilde\bfY^{nj} + \bfb^{nj} , \qquad j=1,\dotsc,s .
		\end{aligned}
	\end{equation}
	
	The errors, defined by
	\begin{equation*}
		\bfe^n = \bfy^n - \widetilde\bfy^n, \quad \bfE^{nj} = \bfY^{nj} - \widetilde\bfY^{nj}, \andquad \dot\bfE^{nj} = \dot\bfY^{nj} - \dot{\widetilde{\bfY}}^{nj} ,
	\end{equation*}
	satisfy the error equations (obtained by subtracting \eqref{eq:RK defect def} from \eqref{RK method}):
	\begin{subequations}
	\label{eq:RK error eq}
	\begin{align}
		\label{eq:RK error eq a}
		\bfE^{ni}  = &\ \bfe^n + \tau \sum_{j=1}^n a_{ij} \dot\bfE^{nj} - \bfD^{nj}, \qquad i=1,\dotsc,s, \\
		\label{eq:RK error eq b}
		\bfe^{n+1}  = &\ \bfe^n + \tau \sum_{j=1}^n b_j \dot\bfE^{nj} - \bfd^{n+1} \\
		\label{eq:RK error eq c}
		\dot\bfE^{nj} = &\ \bfJ\inv (\bfH + \bfHH) \bfE^{nj} , \qquad j=1,\dotsc,s .
	\end{align}
	\end{subequations}

	(a) \emph{Local error.} We first estimate the local error, i.e.\ the error after one step starting from the exact initial value ($\bfe^n=0$). 
	
	The error equation for the internal stages \eqref{eq:RK error eq a} is rewritten using the vectors $\bfE^n = (\bfE^{n1}, \dotsc, \bfE^{ns})^T$ and $\bfD^n = (\bfD^{n1}, \dotsc, \bfD^{ns})^T$:
	\begin{align}
	\label{eq:RK error eq a - vector form}
		\bfE^{n}  = &\ \tau (\AC \otimes \Id) \dot\bfE^{n} - \bfD^{n} .
	\end{align}
	
	We multiply both sides by $(\bfE^n)^T (\calD \AC\inv \otimes \bfS)$ and obtain
	\begin{equation}
	\label{eq:local err tested}
		(\bfE^n)^T (\calD \AC\inv \otimes \bfS) \bfE^{n} = \tau (\bfE^n)^T (\calD \otimes \bfS) \dot\bfE^{n} + (\bfE^n)^T (\calD \AC\inv \otimes \bfS) \bfD^{n},
	\end{equation}
	then we estimates these terms separately.
	
	For the left-hand side the coercivity of the Runge--Kutta method \eqref{eq:R-K coercivity} (recall that $d_i > 0$) yields
	\begin{align*}
		(\bfE^n)^T (\calD \AC\inv \otimes \bfS) \bfE^{n} \geq &\ \alpha (\bfE^n)^T (\calD \otimes \bfS) \bfE^{n} \\
		\geq &\ \alpha \min_{i=1,\dotsc,s}\{d_i\} \sum_{j=1}^s \|\bfE^{nj}\|_\bfS^2 = c_0 \sum_{j=1}^s \|\bfE^{nj}\|_\bfS^2 .
	\end{align*}
	
	For the first term on right-hand side, using \eqref{eq:RK error eq c} and the bounds \eqref{eq:S product estimate} and \eqref{eq:S product estimate with hat H}, we obtain
	\begin{align*}
		(\bfE^n)^T (\calD \otimes \bfS) \dot\bfE^{n} = &\ \sum_{j=1}^s d_j (\bfE^{nj})^T \bfS \bfJ\inv (\bfH + \bfHH) \bfE^{nj} 
		\leq c \sum_{j=1}^s \|\bfE^{nj}\|_\bfS^2 .
	\end{align*}
	
	For the other term on the right-hand side we use Cauchy--Schwarz and Young's inequality we obtain
	\begin{align*}
		(\bfE^n)^T (\calD \AC\inv \otimes \bfS) \bfD^{n} \leq &\ c \sum_{j=1}^s \|\bfE^{nj}\|_\bfS \|\bfD^{nj}\|_\bfS 
		\leq  \frac{c_0}{4} \sum_{j=1}^s \|\bfE^{nj}\|_\bfS^2 + c \sum_{j=1}^s \|\bfD^{nj}\|_\bfS^2 .
	\end{align*}
	
	The combination of these estimates and absorptions (using a sufficiently small $\tau$) yields
	\begin{equation}
	\label{eq:estimate for internal stages - loc err}
		\sum_{j=1}^s \|\bfE^{nj}\|_\bfS^2 \leq c \sum_{j=1}^s \|\bfD^{nj}\|_\bfS^2 .
	\end{equation}
	
	Similarly, \eqref{eq:RK error eq b} can also be rewritten, by expressing $\dot\bfE^n$ from \eqref{eq:RK error eq a - vector form}, as
	\begin{equation*}
		\bfe^{n+1} = \tau b^T \dot\bfE^n - \bfd^{n+1} = (b^T \AC\inv \otimes \Id) (\bfE^n + \bfD^n) - \bfd^{n+1},
	\end{equation*}
	which is estimated, again using Cauchy--Schwarz and Young's inequalities and the bound \eqref{eq:estimate for internal stages - loc err}, by
	\begin{align*}
		\|\bfe^{n+1}\|_\bfS \leq c \sum_{j=1}^s \|\bfD^{nj}\|_\bfS + c \|\bfd^{n+1}\|_\bfS .
	\end{align*}

	(b) \emph{Error propagation.} We analyse how the error of two numerical solutions propagate along the time steps, i.e.\ we study the error equations \eqref{eq:RK error eq} with zero defects ($\bfD^{nj}=0$ and $\bf^{n+1}=0$).
	
	We compute the norm of $\bfe^{n+1}$, which is expressed using \eqref{eq:RK error eq b}, and then rewritten using \eqref{eq:RK error eq a}:
	\begin{equation}
	\label{eq:error propagation first bound}
		\begin{aligned}
			\|\bfe^{n+1}\|_\bfS^2 = &\ \|\bfe^n\|_\bfS^2 + 2 \tau \sum_{j=1}^n (\bfE^{nj})^T \bfS \dot\bfE^{nj} 
			- \tau^2 \sum_{i,j=1}^n (b_i a_{ij} + b_j a_{ji} - b_ib_j) (\dot\bfE^{nj})^T \bfS \dot\bfE^{ni} \\
			\leq &\ \|\bfe^n\|_\bfS^2 + 2 \tau \sum_{j=1}^n (\bfE^{nj})^T \bfS \dot\bfE^{nj}  
		\end{aligned}
	\end{equation}
	The third term on the right-hand side is non-positive due to algebraic stability. The second term is estimated using Cauchy--Schwarz and Young's inequalities, similarly as in part (a), by \eqref{eq:RK error eq c} and the bounds  \eqref{eq:S product estimate} and \eqref{eq:S product estimate with hat H}, we obtain
	\begin{equation}
	\label{eq:second term est error prop}
		(\bfE^{nj})^T \bfS \dot\bfE^{nj} \leq c \|\bfE^{nj}\|_\bfS^2 .
	\end{equation}
	
	Similarly as in part (a) of the proof (with $\bfe^n$ playing the role of $\bfD^{ni}$), we have the estimate
	\begin{equation}
	\label{eq:internal stages est error prop}
		\sum_{j=1}^s \|\bfE^{nj}\|_\bfS^2 \leq c \|\bfe^n\|_\bfS^2 .
	\end{equation}
	
	Altogether, by substituting the bounds \eqref{eq:second term est error prop} and \eqref{eq:internal stages est error prop} into \eqref{eq:error propagation first bound} we obtain
	\begin{align*}
		\|\bfe^{n+1}\|_\bfS^2 \leq (1 + C \tau) \|\bfe^n\|_\bfS^2 .
	\end{align*}
	
	(c) \emph{Error accumulation.} A standard application of Lady Windermere's fan, see \cite[II.3 and I.7]{HairerWannerII}, completes the proof of stability.
\end{proof}

\subsubsection{Convergence}

Via the above stability bound and by the estimates for the semi-discrete defect $d_h$ (defined in \eqref{eq:defect definition}) from Lemma~\ref{lemma:abstract defect bound} (after using the suitable estimates from Section~\ref{section:FEM error analysis}), we obtain the following fully discrete convergence estimates with the stage order $s$ in time. The theorem holds for the general case as long as a defect bound in the dual norm is known. The result is stated simultaneously for all four exemplary cases considered in the paper.
\begin{theorem}
	\label{theorem: convergence - stage order}
	Let the solution of the wave equation with dynamic boundary conditions be sufficiently regular in time $u \in H^{s+2}(0,T;H^2(\Om))$ with $\ga u \in H^{s+2}(0,T;H^2(\Ga))$ (in addition to the regularity assumptions of Theorems~\ref{theorem: semi-discrete error bound: pure second}--\ref{theorem: semi-discrete error bound: acoustic boundary conditions}). Then there is a $\tau_0 >0$ and an $h_0 > 0$ such that, for $\tau \leq \tau_0$ and $h\leq h_0$, the error between the solution $u(\cdot,t_n)$ and the fully discrete solution $(u_h^n)^\ell$, obtained using first order finite elements and an $s$-stage Gauss--Runge--Kutta method, satisfies the following convergence estimates, for $n \tau \leq T$,
	\begin{equation}
	\label{eq:convergence bound with s}
	\|(u_h^n)^\ell - u(\cdot,t_n)\|_{L^2(\Om)} + \|\ga((u_h^n)^\ell - u(\cdot,t_n))\|_{L^2(\Ga)} \leq C (h^k + \tau^s),
	\end{equation}
	where the power $k$ is given in Theorems~\ref{theorem: semi-discrete error bound: pure second}, \ref{theorem: semi-discrete error bound: advective}, \ref{theorem: semi-discrete error bound: strong damping} and \ref{theorem: semi-discrete error bound: acoustic boundary conditions}, for the different problems, respectively. The constant $C>0$ is independent of $h$, $\tau$ and $n$, but depends on the corresponding Sobolev norms of the exact solution $u$ and on $T$.
	
	For purely second order problems with $\beta=0$ the error on the boundary is measured in the $H^{-1/2}(\Ga)$ norm instead of the $L^2(\Ga)$ norm, cf.~\eqref{eq:error estimate - pure second - beta = 0}. 
	
	For problems with acoustic boundary conditions in the second term we have $\delta$ on the boundary instead of $\ga u$, cf.~\eqref{eq:error estimate: acoustic bc}.
\end{theorem}

\begin{proof}
	Similarly as in Section~\ref{section: error analysis and Ritz}, the fully discrete error between the solution $u(\cdot,t_n)$ and the fully discrete solution $u_h^n$ (with nodal values $\bfu^n$, the second component of $\bfy^n$) is decomposed as, with $e_h(t_n)=\Ritz u(\cdot,t_n) -  u_h^n$,
	\begin{equation*}
	\begin{aligned}
	\normH{u(\cdot,t_n) - (u_h^n)^\ell} 
	&\leq   \normH{ u(\cdot,t_n) - \Ritzl u(\cdot,t_n) } + \normH*{\lift{(\Ritz u(\cdot,t_n) -  u_h^n)}} \\
	&\leq   \normH{ u(\cdot,t_n) - \Ritzl u(\cdot,t_n) } + c \normHh*{e_h(t_n)} .
	\end{aligned}
	\end{equation*}
	The error of the Ritz map has been estimated before in Lemma~\ref{lemma: Ritz error est} as $\Landau (h^2)$.
	
	In order to bound the second term we use the stability bound of Lemma~\ref{lemma:RK stability}. This part of the proof is analogous to the proof of \cite[Theorem~5.2]{Mansour_GRK}, however due to some differences we carry it out below.
	
	The vector $\bfe_u\t$ collecting the nodal values of the semi-discrete error $e_h\t$, from Section~\ref{section: error analysis and Ritz}, satisfies the ODE
	\begin{equation}
	\label{eq:ODE with semi-discrete residual}
	\bfM \ddot\bfe_u\t + \bfB \dot\bfe_u\t + \bfA \bfe_u\t = - \overline\bfr\t ,
	\end{equation}
	with the vector $\overline\bfr\t \in \R^N$ collecting the nodal values of the semi-discrete residual $d_h\t \in V_h$ satisfying the equality \eqref{eq:defect definition}.
	%
	
	The error equation is again rewritten as a first order ODE system, collecting the two errors $\bfe\t = (\bfe_v\t,\bfe_u\t)^T$ (here $\bfe_v$ denotes the error in the $\bfv$ component, cf.~\eqref{eq:matrix-vector form - first order}),  which satisfies an ODE system similar to \eqref{eq:first order form}, with $\bfr\t=(\overline\bfr\t,0)^T$,
	\begin{equation}
	\label{eq:first order form - semi-discrete residual}
	\dot\bfe\t = \bfJ\inv \big(\bfH + \bfHH \big) \bfe\t - \bfr\t .
	\end{equation}
	
	
	The fully discrete errors $\bfe^n$ and $\bfE^{nj}$ 
	then satisfy the error equations for the Runge--Kutta method:
	\begin{subequations}
		\label{eq:RK error eq with residual}
		\begin{align}
		\label{eq:RK error eq with residual a}
		\bfE^{ni}  = &\ \bfe^n + \tau \sum_{j=1}^n a_{ij} \bfJ\inv (\bfH + \bfHH) \bfE^{nj} - \Big(\tau \sum_{j=1}^n a_{ij} \bfr^{nj} + \bfD^{nj}\Big), \qquad i=1,\dotsc,s, \\
		\label{eq:RK error eq with residual b}
		\bfe^{n+1}  = &\ \bfe^n + \tau \sum_{j=1}^n b_j \bfJ\inv (\bfH + \bfHH) \bfE^{nj} - \Big(\tau \sum_{j=1}^n b_j \bfr^{nj} + \bfd^{n+1}\Big) .
		\end{align}
	\end{subequations}
	The stability bound from Lemma~\ref{lemma:RK stability} yields
	\begin{equation}
	\|\bfe^n\|_\bfS \leq C \Big(\|\bfe^0\|_\bfS^2 + \tau^2 \sum_{k=0}^{n-1} \sum_{j=1}^s \|\bfr^{kj}\|_\bfS^2 + \sum_{k=0}^{n-1} \sum_{j=1}^s \|\bfD^{kj}\|_\bfS^2 + \sum_{k=1}^n \|\bfd^{k}\|_\bfS^2 \Big)^{1/2}.
	\end{equation}
	
	The temporal defect terms on the right-hand side are estimated separately. The defects related to time discretisations, using Taylor expansion, satisfy
	\begin{align*}
	\bfd^{n+1} =&\ \tau^s \int_{t_n}^{t_{n+1}} K\big((t-t_n)/\tau\big)\widetilde{\bfy}^{(s+1)}\t \d t ,\\
	\bfD^{nj} =&\ \tau^s \int_{t_n}^{t_{n+1}} K_j\big((t-t_n)/\tau\big)\widetilde{\bfy}^{(s+1)}\t \d t ,
	\end{align*}
	with bounded Peano kernels $K$ and $K_j$, see \cite[Section~3.2.6]{Gautschi}, and where the vector $\widetilde{\bfy}\t = (\widetilde{\bfv}\t,\widetilde{\bfu}\t)^T$ is the nodal vector corresponding to
	\begin{equation}
	\label{eq:exact solution vector for Gauss-R-K}
	R_h \dot u(\cdot,t) + B R_h u(\cdot,t) = \sum_{j=1}^N \widetilde{\bfv}_j\t \phi_j \andquad R_h u(\cdot,t) = \sum_{j=1}^N \widetilde{\bfu}_j\t \phi_j .
	\end{equation}
	
	Therefore, by \eqref{eq:S norm def}, norm equivalences and Lemma~\ref{lemma: Ritz error est}, we obtain
	\begin{align*}
	\|\bfd^{n+1}\|_\bfS + \sum_{j=1}^s \|\bfD^{nj}\|_\bfS 
	\leq &\ c \tau^s \int_{t_n}^{t_{n+1}} \|R_h u^{(s+2)}(\cdot,t) + B R_h u^{(s+1)}(\cdot,t)\|_{\star} + |R_h u^{(s+1)}(\cdot,t)| \d t \\
	\leq &\ c \tau^s \int_{t_n}^{t_{n+1}} |R_h u^{(s+1)}(\cdot,t)| + \|R_h u^{(s+1)}(\cdot,t)\| + |R_h u^{(s+2)}(\cdot,t)| \d t \\
	\leq &\ c \tau^s \sum_{i=1}^2 \int_{t_n}^{t_{n+1}} \Big( \|u^{(s+i)}\t\|_{H^2(\Om)}^2 + \|\ga (u^{(s+i)}\t) \|_{H^2(\Ga)}^2 \Big)^{1/2} \d t .
	\end{align*}
	
	For the semi-discrete residual we have, since $\bfr\t=(\overline\bfr\t,0)^T$ and by \eqref{eq:S norm def},
	\begin{align*}
	\tau \sum_{k=0}^{n-1} \sum_{j=1}^s \|\bfr^{kj}\|_\bfS 
	= &\ \tau \sum_{k=0}^{n-1} \sum_{j=1}^s \|\overline\bfr^{kj}\|_{\bfA\inv} 
	= \tau \sum_{k=0}^{n-1} \sum_{j=1}^s \|d_h(t_n+c_j \tau)\|_{\star,h} ,
	\end{align*}
	which was estimated in Lemma~\ref{lemma:abstract defect bound} and Section~\ref{section:FEM error analysis} as $\Landau (h^k)$ with the appropriate $k$ from Theorem~\ref{theorem: semi-discrete error bound: pure second}, \ref{theorem: semi-discrete error bound: advective}, \ref{theorem: semi-discrete error bound: strong damping}, or \ref{theorem: semi-discrete error bound: acoustic boundary conditions}.
	
	Therefore, by \eqref{eq:S norm def} and the above bounds, we have 
	\begin{align*}
	|e_h|_h \leq &\ \|\bfe^n\|_\bfS \leq C ( h^k + \tau^s ) .
	\end{align*}
	This, together with Ritz map error estimates, finishes the proof.
\end{proof}

As in \cite{LubichOstermann_interior} and \cite{Mansour_GRK,rkkato}, under stronger regularity assumptions temporal convergence with the classical order $p$ can also be shown. For Gauss--Runge--Kutta methods the classical order $p = 2s$, see, e.g.~\cite{HairerWannerII}.

We assume that, for the nodal values of the Ritz map of the exact solution $\widetilde{\bfy}\t$ (see \eqref{eq:exact solution vector for Gauss-R-K})
\begin{equation}
\label{eq:more regularity}
\begin{aligned}
\|\bfJ\inv (\bfH + \bfHH)^{k_j-1} \dotsb \bfJ\inv (\bfH + \bfHH)^{k_1-1} \widetilde{\bfy}^{(l)}\t \|_\bfS \leq &\ C_0 , \\
\|\bfJ\inv (\bfH + \bfHH)^{s} \bfJ\inv (\bfH + \bfHH)^{k_j-1} \dotsb \bfJ\inv (\bfH + \bfHH)^{k_1-1} \widetilde{\bfy}^{(l)}\t \|_\bfS \leq &\ C_0 ,
\end{aligned}
\end{equation}
for $s \geq 2$ with a $C_0 > 0$, for all $0 \leq k_i \leq s-1$ and $l \geq s+1$ with $k_1 + \dotsb + k_j + l \leq 2s+1$, where negative powers of operators are understood as the identity operator.

\begin{theorem}
	\label{theorem: convergence - classical order}
	Let the solution of the wave equation with dynamic boundary conditions satisfy the regularity conditions of Theorem~\ref{theorem: convergence - stage order} and additionally those in \eqref{eq:more regularity}. Then there is a $\tau_0 >0$ and an $h_0 > 0$ such that, for $\tau \leq \tau_0$ and $h\leq h_0$, the error between the solution $u(\cdot,t_n)$ and the fully discrete solution $(u_h^n)^\ell$, obtained using first order finite elements and an $s$-stage Gauss--Runge--Kutta method, satisfies the following convergence estimates of classical order $p = 2s$, for $n \tau \leq T$,
	\begin{equation}
	\label{eq:convergence bound with p}
	\|(u_h^n)^\ell - u(\cdot,t_n)\|_{L^2(\Om)} + \|\ga((u_h^n)^\ell - u(\cdot,t_n))\|_{L^2(\Ga)} \leq C (h^k + \tau^{2s}),
	\end{equation}
	where the power $k$ is the same as in Theorem~\ref{theorem: convergence - stage order}. The constant $C>0$ is independent of $h$, $\tau$ and $n$, but depends on the solution $u$, on $C_0$ from \eqref{eq:more regularity}, and on $T$.
	
	Along the same remarks from Theorem~\ref{theorem: convergence - stage order} for $\beta=0$, and for problems with acoustic boundary conditions.
\end{theorem}
\begin{proof}
	The proof of this theorem directly follows the proof of \cite[Theorem~1]{LubichOstermann_interior} (parabolic problems), \cite[Theorem~5.3]{Mansour_GRK} (wave equations), where one additional order is gained by studying the modified error equations with the modified solution $\widehat\bfY^{nj} = \widetilde\bfY^{nj} + \bfD^{nj}$, using the stability bound from Lemma~\ref{lemma:RK stability} and the proof of Theorem~\ref{theorem: convergence - stage order}. The process can be iterated until convergence with classical order is achieved.
\end{proof}

\section{Numerical experiments}
\label{section:numerics}

In this section, we report on numerical experiments which illustrate that the proven spatial and temporal convergence rates of Theorem~\ref{theorem: semi-discrete error bound: pure second}--\ref{theorem: semi-discrete error bound: acoustic boundary conditions} and \ref{theorem: convergence - classical order} are indeed observed (with the exception of strongly damped problems).

We implemented a finite element discretization of the wave equation with dynamic boundary conditions in FEniCS, cf.~\cite{FEniCS15}, while for problems with acoustic boundary conditions we have used a Matlab implementation based on the P2Q2Iso2D code provided by \cite{BarCH06}. The triangulation of the domain was computed using the DistMesh package by \cite{PerS04}.

For all our numerical experiments we have used bulk--surface finite elements using piecewise linear basis functions as a space discretisation and the Gauss--Runge--Kutta method with one node ($s=1$) and of order two, i.e.~the implicit midpoint rule, for time integration. For each test problem the numerical solutions were computed for a sequence of time step sizes $\tau_j = \tau_{k-1}/2$ with $\tau_0 = 2^{-5}$ and a sequence of meshes with mesh widths $h_j \approx h_{k-1}/2$ with $h_0 \approx 0.33$.

In the case of numerical experiments when the exact solution is not known, the errors shown in the figures are obtained by comparing the numerical solution with a reference solution, which is computed using quadratic isoparametric elements \bbk (i.e.~using a mesh with a quadratic approximation of the boundary and quadratic basis functions) \ebk on the finest mesh (for FEniCS simulation on the second finest mesh) and using the smallest time step size from above. Otherwise the exact and numerical solutions are compared. In the figures we plotted the errors at time $T = 1$, while for acoustic boundary conditions at $T=0.2$.

All tests were carried out on the two dimensional unit disc and its boundary:
\begin{align*}
\Omega = \{ x \in \R^2 \mid \abs{x} < 1 \}, \andquad \Ga = \pa \Om = \{ x \in \R^2 \mid \abs{x} = 1 \}.
\end{align*}

\subsection{Purely second-order dynamic boundary conditions -- Theorem~\ref{theorem: semi-discrete error bound: pure second} and \eqref{theorem: convergence - classical order}}
For our first test, we consider the wave equation with purely second-order dynamic boundary condition \eqref{wave eqn - pure second-order} with $\mu = \beta = 1$, \bdh $\kappa = 0$ and \edh $f_\Om = f_\Ga = 0$.
The initial values are
\begin{align}
\label{eq:iv_func}
u(x,0) = e^{-20((x_1-1)^2+x_2^2)}, \andquad \dot u(x,0) = 0,
\end{align}
such that the solution shows a surface wave travelling along $\Gamma$ due to the dynamic boundary condition.

\begin{figure}[hbp]
	\centering
	\includegraphics[width=1\linewidth]{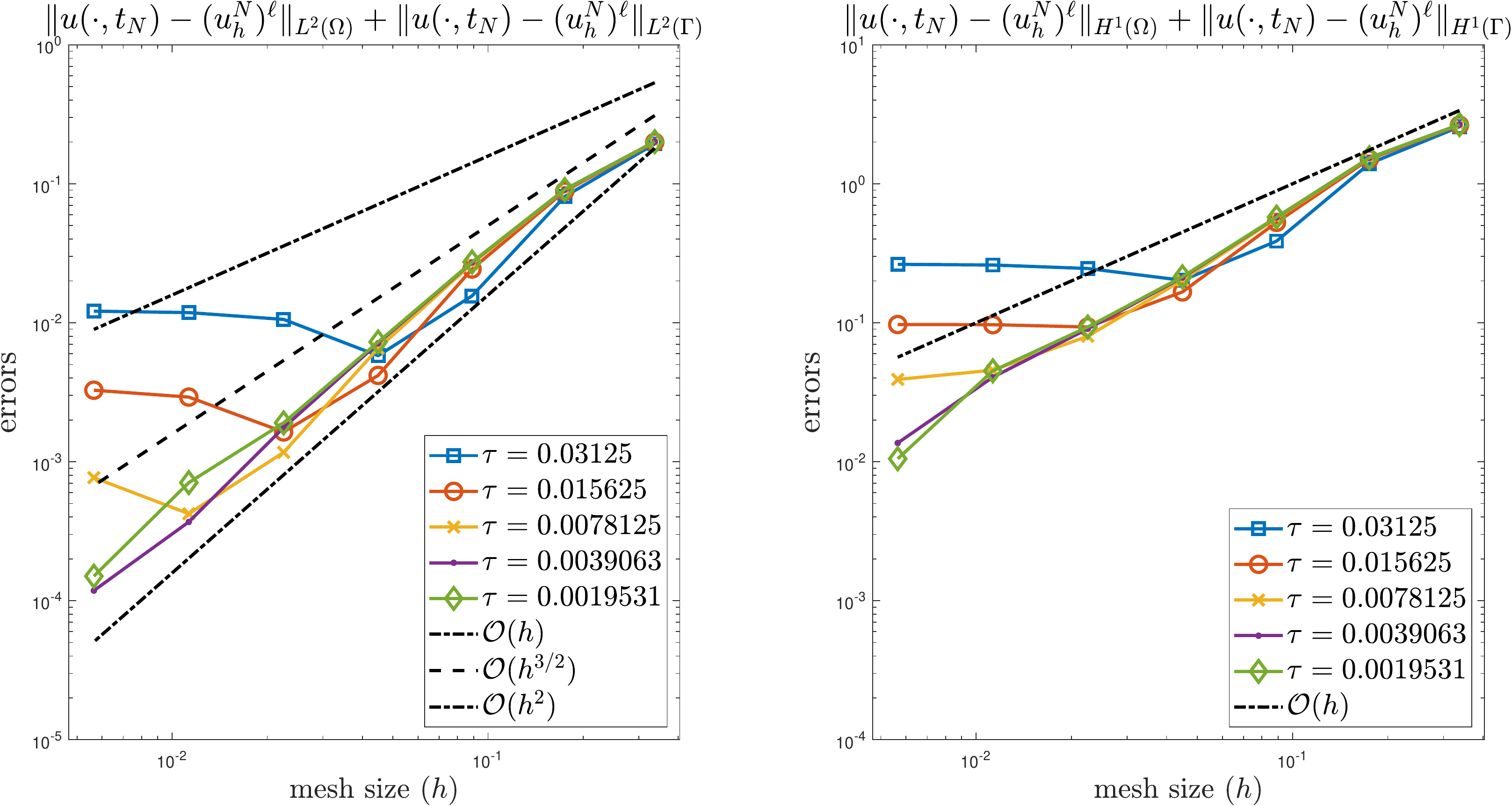}
	\caption{Spatial convergence plots for purely second order problems (Theorem~\ref{theorem: semi-discrete error bound: pure second})}
	\label{figure:purely second order spatial}
\end{figure}
\begin{figure}[hbp]
	\centering
	\includegraphics[width=0.5\linewidth]{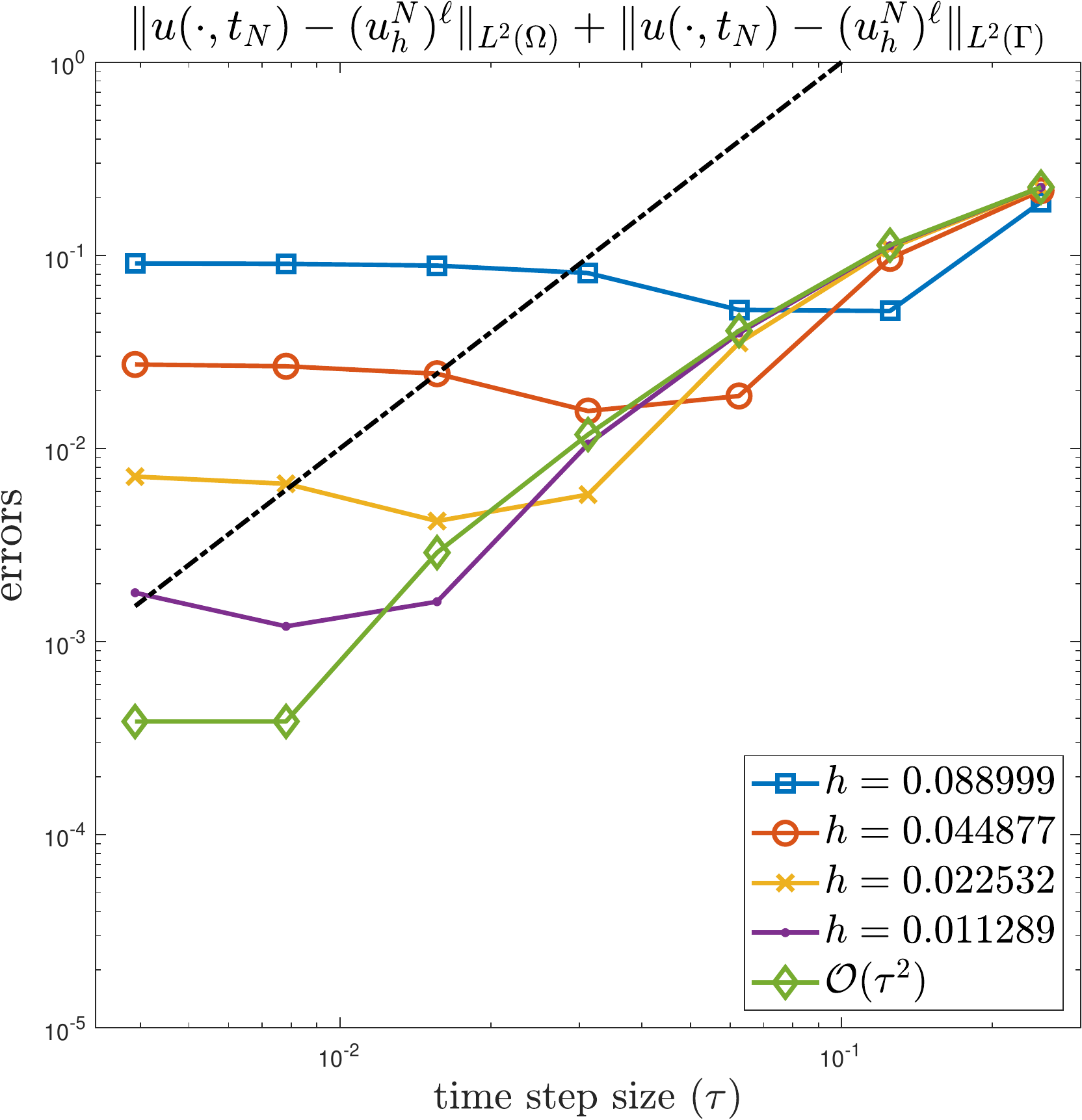}
	\caption{Temporal convergence plots for purely second order problems (Theorem~\ref{theorem: convergence - classical order})}
	\label{figure:purely second order temporal}
\end{figure}

The logarithmic plots show the errors, in the $L^2$ and $H^1$ norms, against the mesh width $(h_j)$ in Figure~\ref{figure:purely second order spatial}, and the error in the $L^2$ norm against the time step size $(\tau_j)$ in Figure~\ref{figure:purely second order temporal}. As shown by Figure~\ref{figure:purely second order spatial} and \ref{figure:purely second order temporal} the $\Landau(h^2)$ spatial and $\Landau(\tau^2)$ temporal convergence rates, respectively, are in agreement with the theoretical convergence results.
\bdh The errors in the energy norm (i.e.~bulk-surface $H^1$ norm) is plotted in Figure~\ref{figure:purely second order spatial} to allow easy comparison with the results of \cite{Hip17}, and thereby validate the setup of the chosen test problem. 
The error lines with different markers correspond to different time steps or mesh refinements, respectively. \edh

In Figure~\ref{figure:purely second order spatial} (in both plots) we can observe two regions: a region where the spatial discretisation error dominates, matching to the $\Landau(h^2)$ order of convergence of our theoretical results (the error curves are parallel to a reference line), and a region, with small time step sizes, where the temporal discretization error dominates (the error curves flatten out). 

\bbk In Figure~\ref{figure:purely second order temporal} we report on the temporal convergence rates for the problem \eqref{wave eqn - pure second-order} with purely second order dynamic boundary conditions, i.e.~the above description applies with reversed roles: Now the temporal convergence rate matches $\Landau(\tau^2)$ from Theorem~\ref{theorem: convergence - classical order} in case of the implicit midpoint rule ($s=1$) \bdh until the spatial error dominates. \edh 
The temporal convergence behaviour of the other problems is very similar as the one presented here, therefore those plots are omitted. \ebk 

\subsection{Advective dynamic boundary conditions -- Theorem~\ref{theorem: semi-discrete error bound: advective}}

\bdh

For the tests with advection terms, we consider the wave equation with advective dynamic boundary conditions \eqref{eq:wave eqn - advective} in two setups.

First we consider \eqref{eq:wave eqn - advective} with constant advection in the bulk $\advB(x,t)=(2,0)^T$, and $\advS(x,t)=(0,0)^T$ on the surface, and use the same coefficients $\mu = \beta = 1$, $\kappa = 0$, and right-hand side $f_\Om = f_\Ga = 0$ and initial values \eqref{eq:iv_func} as before.
The plots in Figure~\ref{figure:advection} show that the finite element approximation of wave equations dynamic boundary conditions \emph{with} bulk advection ($\advB\neq0$) converges with $\Landau(h^{3/2})$ (note the reference lines) which was proved in Theorem~\ref{theorem: semi-discrete error bound: advective}.
\begin{figure}[htp]
	\centering
	\includegraphics[width=\linewidth]{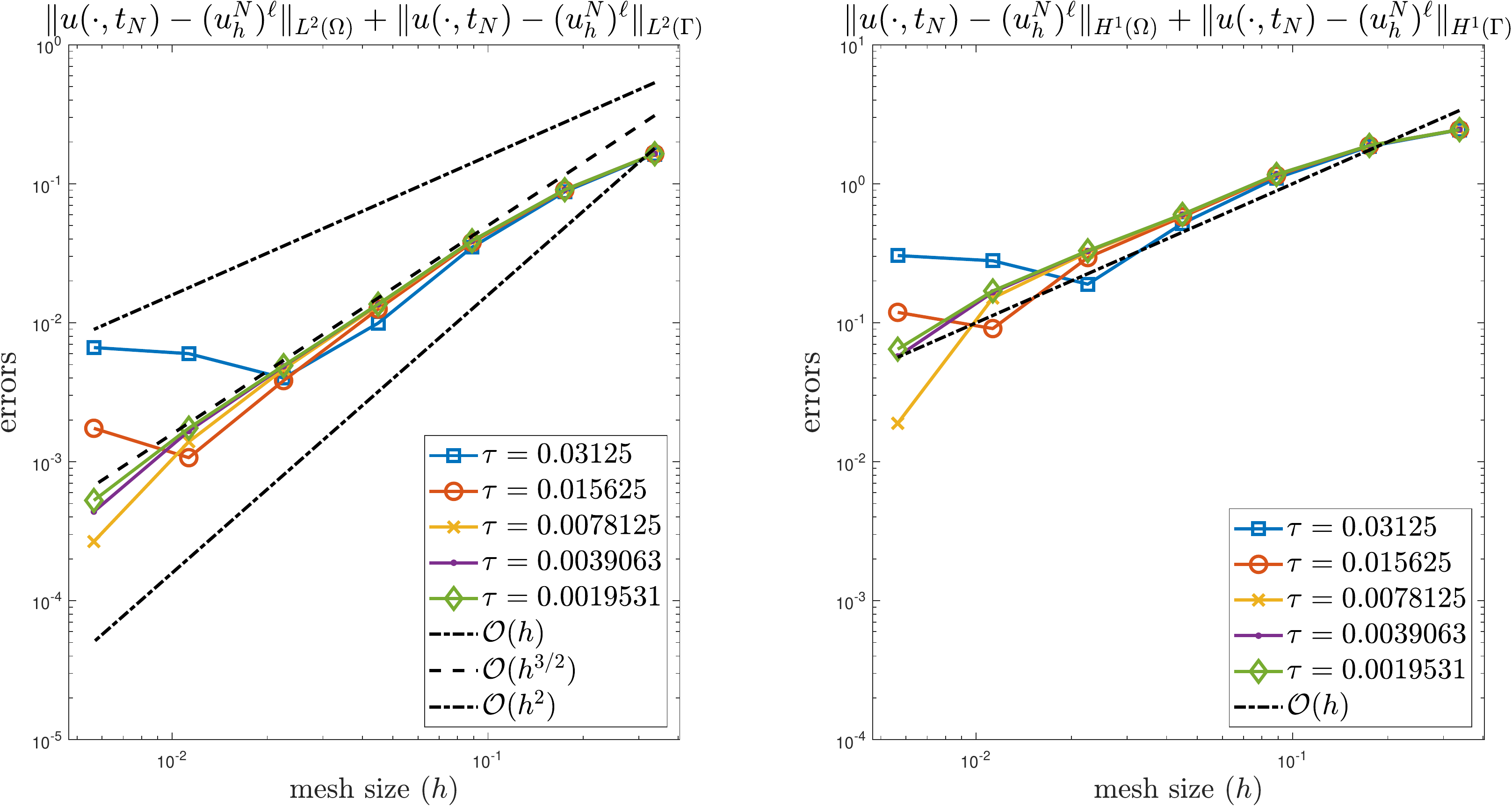}
	\caption{Spatial convergence plots for problems with advective dynamic boundary conditions (Theorem~\ref{theorem: semi-discrete error bound: advective})}
	\label{figure:advection}
\end{figure}

Second we consider the same problem with dynamic boundary conditions but with advection \emph{only} on the surface: $\advB(x,t) = (0,0)^T$ and $\advS(x,t)=(-x_2,x_1)^T$. 
The plots in Figure~\ref{figure:only surface advection} show $\Landau(h^2)$ which is in agreement with our theoretical results from Theorem~\ref{theorem: semi-discrete error bound: advective}. \edh
\begin{figure}[htbp]
	\centering
	\includegraphics[width=\linewidth]{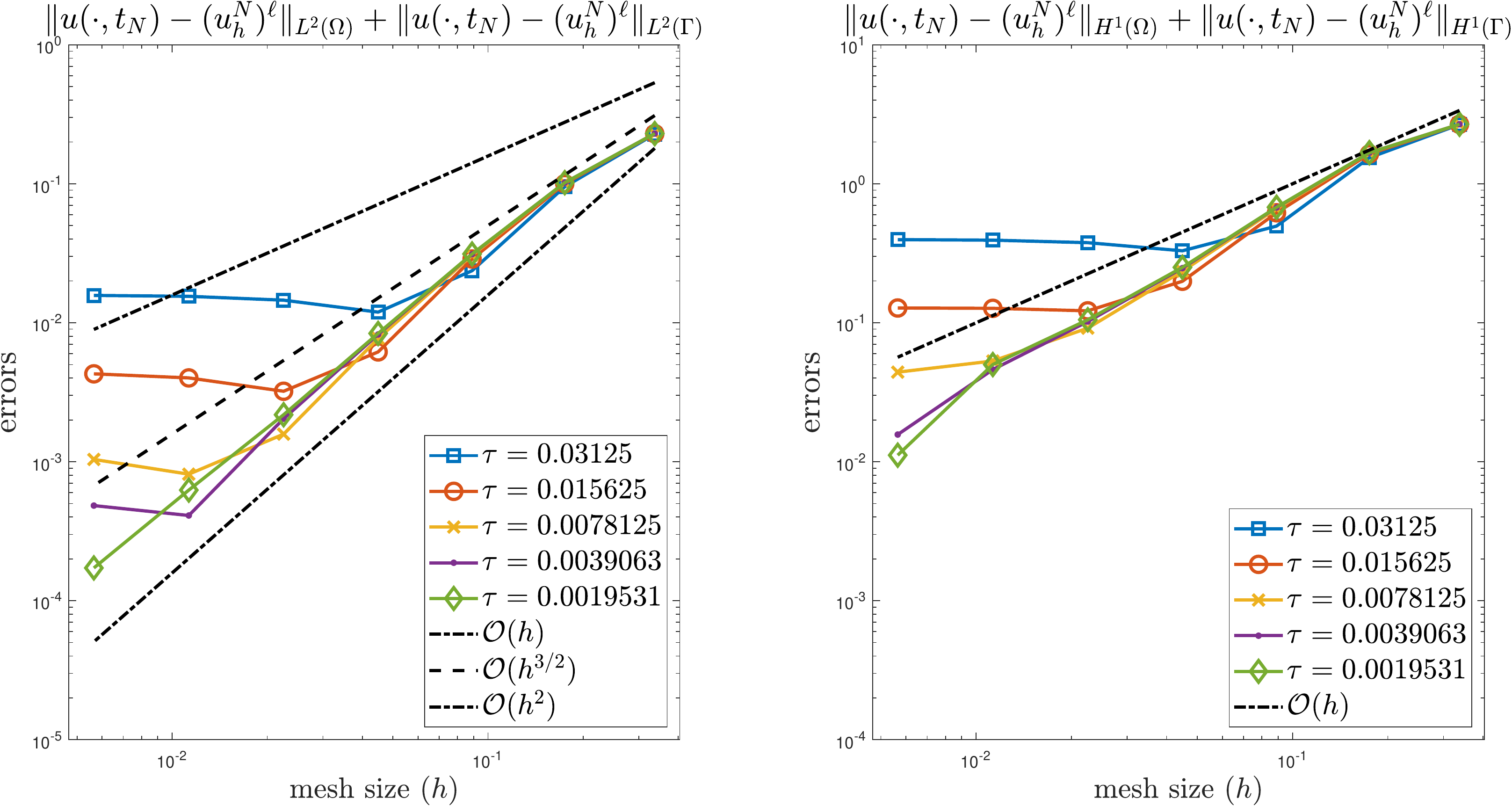}
	\caption{Spatial convergence plots for problems with advective dynamic boundary conditions (Theorem~\ref{theorem: semi-discrete error bound: advective})}
	\label{figure:only surface advection}
\end{figure}

\subsection{\Sdamp\ dynamic boundary conditions -- Theorem~\ref{theorem: semi-discrete error bound: strong damping}}

For the tests with strong damping, we consider the wave equation with \sdamp\ dynamic boundary condition \eqref{eq:wave eqn - strong damping} with $\mu = \beta = 1$, \bdh $\kappa = 0$, $f_\Om = f_\Ga = 0$ and damping coefficients $d_\Om=0.1$, $d_\Ga=0.2$. \edh
The initial values are the same as in \eqref{eq:iv_func}.

The plots in Figure~\ref{figure:strong damping} show the same spatial convergence plots as described before, but for the above problem with \sdamp\ dynamic boundary conditions. We note here that for the case $\beta \neq d_\Ga/d_\Om$ the convergence order of Theorem~\ref{theorem: semi-discrete error bound: strong damping} is not observed. The expected optimal second-order convergence rate is illustrated by our numerical experiment, \cf the remark after Theorem~\ref{theorem: semi-discrete error bound: strong damping}.
\begin{figure}[htbp]
	\centering
	\includegraphics[width=\linewidth]{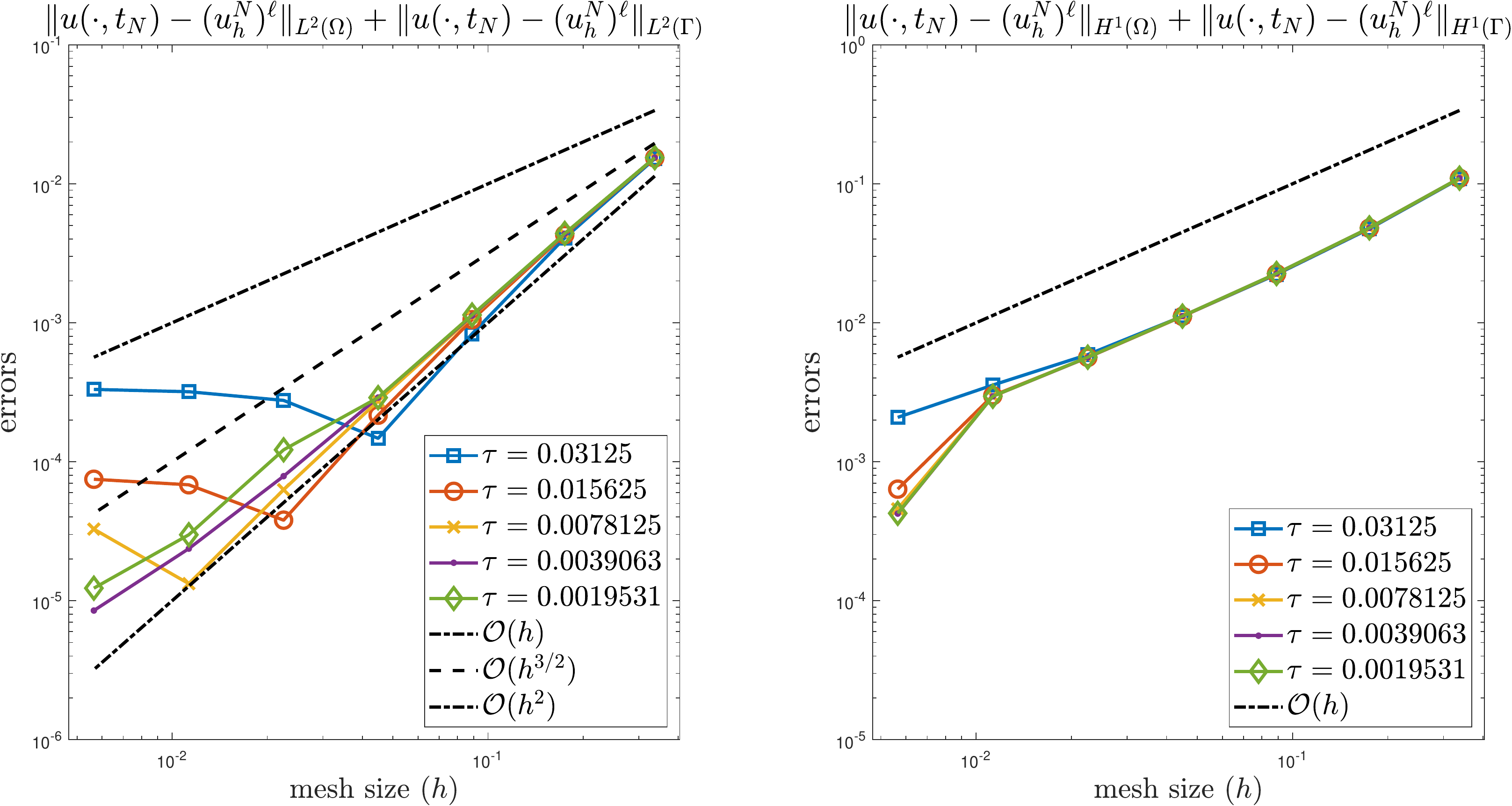}
	\caption{Spatial convergence plots for a problem with \sdamp\ dynamic boundary conditions (Theorem~\ref{theorem: semi-discrete error bound: strong damping})}
	\label{figure:strong damping}
\end{figure}

\subsection{Acoustic boundary conditions -- Theorem~\ref{theorem: semi-discrete error bound: acoustic boundary conditions}}

Finally, we consider the wave equation with acoustic boundary condition \eqref{eqs:ex-acoustic-pde} with unit constants, and with $f_\Om$ and $f_\Ga$ choose such that the exact solution is $u(x,t) = \sin(2\pi t)(x_1^2+x_2^2)^{k/2}$ and $\delta(x,t) = k (2\pi)^{-1}\cos(2\pi t)(x_1^2+x_2^2)^{k/2}$ for $k=1.2$.
The initial values are the interpolations of the exact initial data.

The plots in Figure~\ref{figure:acoustic boundary conditions} show the same spatial convergence plots as described previously, but for the above problem with acoustic boundary conditions. The $\Landau(h^{3/2})$ spatial convergence rates are in agreement with our theoretical results proved in Theorem~\ref{theorem: semi-discrete error bound: acoustic boundary conditions}.
\begin{figure}[hbp]
	\centering
	\includegraphics[width=1\linewidth]{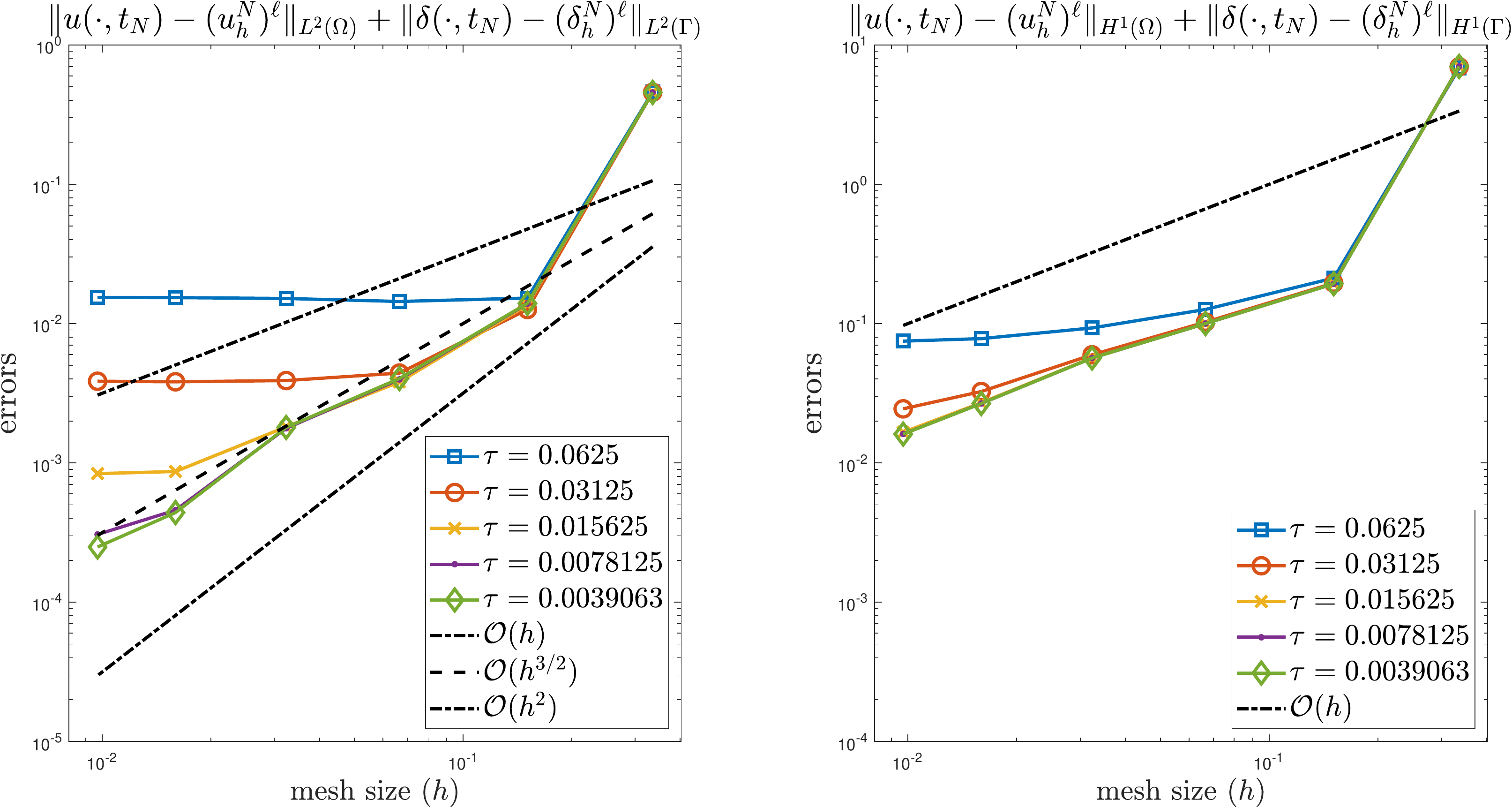}
	\caption{Spatial convergence plots for a problem with acoustic boundary conditions (Theorem~\ref{theorem: semi-discrete error bound: acoustic boundary conditions})}
	\label{figure:acoustic boundary conditions}
\end{figure}

\clearpage

\section{Conclusions}
Albeit the fact that the solutions of wave equations with dynamic boundary conditions have better regularity and stability properties on the boundary than classical Neumann or Dirichlet boundary conditions, our results show that in some cases one actually can not expect $\mathcal O(h^2)$ convergence in the $L^2$ norm.
As our numerical tests show, these reduced convergence rates can actually be observed in simulations and are not due to a crude error analysis.

\section*{Acknowledgement}
We are grateful to Prof.~G.~Zouraris for bringing the fundamental paper of \cite{Fairweather} to our attention.
We also thank Prof.~Marlis Hochbruck and Prof.~Christian Lubich for our invaluable discussions on the topic, and Jan Leibold for careful proofreading.

\bbk We thank two anonymous referees for the comments which helped us to improve on a previous version of this paper. \ebk 

This work is funded by Deutsche Forschungsgemeinschaft, SFB 1173. 


\end{document}